\documentclass[sort&compress,preprint,3p]{elsarticle}
\usepackage[T1]{fontenc}
\usepackage{lmodern}
\biboptions{}
\usepackage{amsthm, amssymb, amsmath, amsfonts, mathtools, thmtools}
\usepackage{tikz-cd, graphicx}
\usetikzlibrary{arrows}
\usepackage{floatrow}
\usepackage[labelformat=empty]{caption}
\usepackage[shortlabels]{enumitem}
\usepackage{bbm}
\usepackage{hyperref}
\usepackage[nameinlink]{cleveref}
\usepackage{stmaryrd} 
\usepackage{todonotes}

\newcommand\myshade{85}
\definecolor{mycitecolor}{rgb}{0.94, 0.8, 0.0}
\definecolor{mylinkcolor}{rgb}{0.56, 0.0, 1.0}
\definecolor{myurlcolor}{rgb}{0.5, 1.0, 0.83}
\hypersetup{
  linkcolor  = mylinkcolor!\myshade!black,
  citecolor  = mycitecolor!\myshade!black,
  urlcolor   = myurlcolor!\myshade!black,
  colorlinks = true,
}

\makeatletter
\def\namedlabel#1#2{\begingroup
    #2%
    \def\@currentlabel{#2}%
    \phantomsection\label{#1}\endgroup
}

\declaretheorem[name=Theorem,refname={Theorem,Theorems},Refname={Theorem,Theorems},numberwithin=section,]{theorem}
\declaretheorem[name=Proposition,refname={Proposition,Propositions},Refname={Proposition,Propositions},sibling=theorem,]{proposition}
\declaretheorem[name=Lemma,refname={Lemma,Lemmas},Refname={Lemma,Lemmas},sibling=theorem,]{lemma}
\declaretheorem[name=Corollary,refname={Corollary,Corollaries},Refname={Corollary,Corollaries},sibling=theorem,]{corollary}

\declaretheoremstyle[  spaceabove=2pt, spacebelow=2pt,
  bodyfont=\normalfont\itshape,
 postheadspace=1em
]{claimnonum}
\declaretheorem[name=Claim,numbered=no,style=claimnonum]{claim*}

\declaretheorem[name=Definition,refname={Definition,Definitions},Refname={Definition,Definitions},sibling=theorem,style=definition,]{definition}
\declaretheorem[name=Notation,refname={Notation,Notations},Refname={Notation,Notations},sibling=theorem,style=definition,]{notation}

\declaretheorem[name=Remark,refname={Remark,Remarks},Refname={Remark,Remarks},sibling=theorem,style=remark,]{remark}
\declaretheorem[name=Remark,style=remark,numbered=no,]{remark*}

\newcommand{\2}{\mathbbm{2}}

\newcommand{\df}{\coloneqq}

\newcommand{\epi}{\twoheadrightarrow}
\newcommand{\repi}{\rightarrowtriangle}
\newcommand{\mono}{\rightarrowtail}
\newcommand{\rmono}{\hookrightarrow}

\newcommand{\C}{\mathsf{C}} 
\newcommand{\MV}{{\mathsf{MV}}} 
\newcommand{\MVlf}{{\mathsf{MV}_\mathsf{lf}}} 
\newcommand{\MS}{\mathsf{MS}} 
\newcommand{\BS}{\mathsf{BS}}
\newcommand{\BA}{\mathsf{BA}}
\newcommand{\Set}{\mathsf{Set}}
\newcommand{\Top}{\mathsf{Top}}

\newcommand{\GG}{\ensuremath{\mathcal G}}
\newcommand{\NN}{\ensuremath{\mathcal N}}

\newcommand{\N}{\mathbb{N}} 
\newcommand{\Nnot}{\N^+} 
\renewcommand{\P}{\mathbb{P}} 
\newcommand{\Q}{\mathbb{Q}}

\newcommand{\Dn}{\mathrm{D}_{n}}
\newcommand{\D}[1]{\mathrm{D}_{{#1}}}

\newcommand{\Quot}{\mathbf{Q}} 
\newcommand{\CS}{\mathbf{R}} 

\newcommand{\seq}{\subseteq}
\newcommand{\eps}{\varepsilon}

\begin{document}

\title{Are locally finite MV-algebras a variety?}
\author[1]{Marco~Abbadini}
\ead{mabbadini@unisa.it}
\author[1]{Luca~Spada\corref{cor1}}
\ead{lspada@unisa.it}

\affiliation[1]{
organization={Dipartimento di Matematica, Universit\`a degli Studi di Salerno}, 
addressline={Piazza Renato Caccioppoli, 2},
postcode={84084}, 
city={Fisciano (SA)}, 
country={Italy}
}
\cortext[cor1]{Corresponding author}

\begin{abstract}
We answer Mundici's problem number 3 (D.\ Mundici.\ \emph{Advanced {\L}ukasiewicz calculus.} Trends in Logic Vol.\ 35.\ Springer 2011, p.\ 235): \emph{Is the category of locally finite MV-algebras equivalent to an equational class?}  We prove:
\begin{enumerate}
\item The category of locally finite MV-algebras is not equivalent to any finitary variety.
\item More is true: the category of locally finite MV-algebras is not equivalent to any finitely-sorted finitary quasi-variety.
\item The category of locally finite MV-algebras is equivalent to an infinitary variety; with operations of at most countable arity.
\item The category of locally finite MV-algebras is equivalent to a countably-sorted finitary variety.
\end{enumerate}
Our proofs rest upon the duality between locally finite MV-algebras and the category of ``multisets'' by R.\ Cignoli, E.\ J.\ Dubuc and D.\ Mundici, and known categorical characterisations of varieties and quasi-varieties.  In fact, no knowledge of MV-algebras is needed, apart from the aforementioned duality.
\end{abstract}
%

\maketitle

\section{Introduction}
\label{s:introduction}
MV-algebras were introduced in \cite{chang1958algebraic} to serve as algebraic semantics for the many-valued {\L}ukasiewicz propositional logic.  They are defined as algebras of the form $\langle A, \oplus, \neg, 0\rangle$ such that $\langle A,\oplus, 0\rangle$ is a monoid, the unary operation $\neg$ satisfies $\neg\neg x=x$ and $\neg 0\oplus x=\neg 0$, and the following axiom holds: $\neg(\neg x\oplus y)\oplus y)=\neg(\neg y\oplus x)\oplus x)$ ---commutativity of $\oplus$ follows from the other axioms \cite{Kolarik2013}.
In their over fifty years of history, MV-algebras have found surprising applications in many fields of mathematics (see, e.g., \cite{cabrer2017classifying,cicalese2011recent,marra2007lebesgue,marra2017stone,mundici2017fans,mundici2018word}); the reader is referred to \cite{zbMATH01382786} for their basic theory and to \cite{Mun2011} for more advanced topics.   Examples of MV-algebras are given by the real interval $[0,1]$ endowed with the operations:
\begin{align}\label{eq:standard}
x\oplus y\df \min\{x+y,1\} \text{ and } \neg x\df 1-x,
\end{align}
as well as  its subalgebra of rational numbers, and the ones generated by any rational number; the latter all take the form $S_{n}\df\{0,\frac{1}{n},\frac{2}{n}\dots,\frac{n-1}{n},1\}$ for some $n\in \Nnot\df\{1,2,3,\dots\}$.  It is important to notice at this point that the MV-operations $\oplus$ and $\neg$ always allow to define a lattice order by setting
\[x\leq y \text{ if and only if } \neg(\neg x\oplus y)\oplus y=y. \]

An MV-algebra is called a \emph{chain} if it is linearly ordered under the above-defined order.  The finite MV-chains are exactly the algebras $S_{n}$ for $n \in \Nnot$ and every finite MV-algebra is a direct product of finite chains \cite[Proposition 3.6.5]{zbMATH01382786}. This simple observation provides an intuitive entryway to a rich front in MV-algebras: their duality theory.  Indeed, any finite MV-chain is completely characterised by the natural number $n$ as above.  Thus, every finite MV-algebra can be recovered from a finite set in which every element has attached a natural number different from 0, i.e.\ from a \emph{finite multiset}.  We call these numbers the denominators (somewhere else called multiplicities) of the elements, for reasons that will be made clear in \cref{r:duality}.  Furthermore, a homomorphism exists from an MV-chain $S_{m}$ to an MV-chain $S_{n}$ if and only if $m$ divides $n$.  Therefore, we will consider maps between multisets that decrease ---in the order of divisibility--- denominators, more precisely: if $(S_{1}, d_{1} \colon S_{1} \to \Nnot)$ and $(S_{2},d_{2}\colon S_{2}\to \Nnot)$ are two finite multisets, a function $f\colon S_{1}\to S_{2}$ is said to \emph{decrease denominators} if, for any $s\in S_{1}$, $d_{2}(f(s))$ divides $d_{1}(s)$. This correspondence is in fact a categorical equivalence between the category $\MV_{\mathsf{f}}$ of finite MV-algebras, together with their homomorphisms,  and the category $\MS_{\mathsf{f}}$ of finite multisets and functions that decrease denominators.

As often happens, the ``toy'' duality described above hints at a more interesting one that is obtained by taking, on the one hand, the completion of $\MV_{\mathsf{f}}$ under all directed co-limits ---called $\mathsf{ind\text{-}}\MV_{\mathsf{f}}$--- and, on the other, the completion of $\MS_{\mathsf{f}}$ under all co-directed limits ---called $\mathsf{pro\text{-}}\MS_{\mathsf{f}}$--- (see \cite[Chapter~6]{Joh86} for more details and applications of this technique).  Indeed, one immediately obtains a formal categorical duality: $\mathsf{ind\text{-}}\MV_{\mathsf{f}}\simeq (\mathsf{pro\text{-}}\MS_{\mathsf{f}})^{\text{op}}$.
Recall that, as a general concept in universal algebra, an algebra $A$ is called \emph{locally finite} if every finitely generated subalgebra of $A$ is finite. In any variety $V$, the locally finite algebras are exactly the direct limits (= directed co-limits) of finite algebras in $V$.  Thus, $\mathsf{ind\text{-}}\MV_{\mathsf{f}}$ is equivalent to the category $\MVlf$ of locally finite MV-algebras and it remains to provide a more familiar description of  $\mathsf{pro\text{-}}\MS_{\mathsf{f}}$.  This is done in \cite{CigDubMun04}, where the authors prove that a certain category $\MS$ of ``multisets'' (with possibly infinite underlying set and possibly infinite denominator function) is the pro-completion of the category of finite multisets (see \cref{d:multiset} for details).

As a matter of fact, some MV-algebras are not locally finite, the MV-algebra $[0,1]$ with the operations defined in \eqref{eq:standard} being a prime example.  Indeed, it can be easily seen that any subalgebra of $[0,1]$ generated by an irrational number is infinite.  Therefore, locally finite MV-algebras are a proper subclass of MV-algebras.  The class of locally finite MV-algebras is easily seen to be closed under homomorphic images, subalgebras, and finite products.  However, it is not closed under arbitrary products.  Nonetheless, the category of locally finite MV-algebras has all products in the categorical sense and they can be described as certain subalgebras of the classical algebraic product \cite[Theorem 5.4]{CigDubMun04}.  Driven by these considerations, in one of the eleven problems at the end of \cite{Mun2011}, D.\ Mundici asks whether the category of locally finite MV-algebras is equivalent to an equational class.  We point out that if such an equivalence exists it is necessarily not concrete, because the underlying set functor from $\MVlf$ to $\Set$ does not preserve products. 

We answer Mundici's question using categorical results that give precise descriptions of the categories that are equivalent to varieties of algebras.  It turns out that Mundici's question can be answered both in the negative and in the positive, depending on the restrictions on the language that one wants to assume.

The study of abstract characterisations of varieties and quasi-varieties has a long history in category theory, starting with the works of M.\ Barr \cite{barr1966acyclic}, J.\ Beck \cite{Beck:1967}, F.\ W.\ Lawvere \cite{Lawvere:1963} and F.\ E.\ J.\ Linton \cite{Linton:1966} in the 1960s and improved by many others. The book \cite{AlgThe} provides an updated account on the subject.

 We conclude this section by briefly describing our results and the techniques we use to prove them.

\Cref{s:preliminaries} includes the definition and some basic information on the category $\MS$. A presentation of the results about algebraic categories needed in the paper is given.

In \Cref{s:finite-language} we prove that the category $\MVlf$ of locally finite MV-algebras is not equivalent to any variety of finitary algebras (\cref{c:locally-finite-are-not-finitary-variety}). Our argument proceeds along the following lines. Every variety of finitary algebras admits a ``finitely generated regular generator'' (see \cref{d:regular-generator}): the free algebra over one generator (see \cref{t:char-quasi-varieties}).
We show that $\MVlf$ does not admit any finitely generated regular generator by proving a dual statement for $\MS$: $\MS$ does not admit any finitely co-generated regular co-generator.  Indeed we prove that, in $\MS$,
\begin{enumerate}[(T1)]
	\item \label{i:fin-co-gen-implies-fin}
		finitely co-generated objects have a finite underlying set (\cref{p:how-finitely-cogenerated});
	\item \label{i:reg-co-gen-implies-infin}
		regular co-generators have an infinite underlying set (\cref{c:reg-cogen}).
\end{enumerate}
\Cref{i:fin-co-gen-implies-fin,i:reg-co-gen-implies-infin} together imply that no object of $\MS$ is simultaneously finitely co-generated and a regular co-generator, thus proving that the opposite category $\MVlf$ does not admit any finitely generated regular generator.  We conclude that $\MVlf$ is not equivalent to any variety of finitary algebras.
Since every quasi-variety of finitary algebras admits a finitely generated regular generator, too, the argument above implies the stronger fact that  $\MVlf$ is not equivalent to any quasi-variety of finitary algebras, either.

In fact, our result is even more general: no finitary language with finitely many sorts affords a (quasi-)equational axiomatisation of $\MVlf$ (\cref{t:locally-finite-are-not-finitely-sorted-quasi-variety}). The ideas are similar: 
Every finitely sorted quasi-variety of finitary algebras admits a regularly generating finite set of finitely generated objects (these objects are ---roughly speaking--- one for each sort: they correspond to the free algebra over an element placed in that particular sort).
We have:
\begin{enumerate}[(T1)]\setcounter{enumi}{2}
	\item \label{i:reg-co-gen-implies-infin-more-sorts}
		a regularly co-generating set either has infinitely many objects or contains an object whose underlying set is infinite (\cref{t:how-reg-cogen}).	
\end{enumerate}
\Cref{i:fin-co-gen-implies-fin,i:reg-co-gen-implies-infin-more-sorts} together imply that $\MS$ admits no regularly co-generating finite set of finitely co-generated objects; thus $\MVlf$ is not equivalent to any finitely sorted quasi-variety of finitary algebras.

In \Cref{s:infinite-language} we exhibit a countable family of regular injective multisets that is a regularly co-generating abstractly co-finite set. Using once more the categorical characterisation outlined in \cref{t:char-quasi-varieties}, we deduce that $\MVlf$ is equivalent to a quasi-variety of finitary algebras in a language with countably many sorts (\cref{p:sorted-quasi-variety}).  At this point, we note that the product of the aforementioned family is a regular injective regular co-generator of $\MS$. We deduce that $\MVlf$ is equivalent to a quasi-variety of algebras in a language with operations of at most countable arity (\cref{p:MVlf-quasi-variety}).

In \Cref{s:from-quasi-variety-to-variety}, we improve the results of the previous section by studying the (internal) equivalence relations of $\MVlf$ (see \cref{d:equivalence-relation} for a short reminder of their definition).
Working once more in the dual category $\MS$, we prove that all reflexive relations in $\MVlf$ are effective equivalence relations;
thus $\MVlf$ is a Mal'tsev category (i.e.\ it has finite limits and every reflexive relation is an equivalence relation) and every equivalence relation in $\MVlf$ is effective.
Using once more the abstract characterisation of \cref{t:char-quasi-varieties} we deduce that $\MVlf$ is equivalent to a variety in a language:
\begin{enumerate}[(T1)]\setcounter{enumi}{3}
\item with one sort and operations of at most countable arity (\cref{t:MVlf-variety}), and 
\item with countably many sorts and operations of finite arity (\cref{t:sorted-variety}).
\end{enumerate} 
Finally, in the Appendix, we offer some supplementary results on the topology of the supernatural numbers (to be defined in \Cref{ss:lf-mv-algebras}) and on internal relations in $\MVlf$.  Although these results are not needed in the proofs outlined above, we think that they might be interesting in the study of the category $\MS$.

\section{Preliminaries}
\label{s:preliminaries}

\begin{notation}
We write $\N$ for the set of natural numbers $\{0,1,2, \dots\}$, $\Nnot$ for $\{1,2,3,\dots\}$ and  $\P$ for the set of prime numbers. Throughout the paper, we assume all categories to be locally small. 
Unless otherwise stated, varieties and quasi-varieties possibly admit operations of infinite arity in their signatures. As customary, we use the prefix ``co-'' for the dual of a categorical concept, obtained by switching the direction of any arrow involved in the definition; we depart from this notation only when the dual notion is traditionally indicated with a different name (e.g., monic and epic).
\end{notation}
As customary, we call \emph{Boolean space} any topological space that is compact, Hausdorff, and has a basis of clopen subsets.  We denote by $\BS$ the category of Boolean spaces with continuous maps and by $\BA$ the category of Boolean algebras and their homomorphisms. In the paper we only assume some basic knowledge on Boolean spaces, Boolean algebras, and their duality; for more information, the reader is referred to \cite{givant2008introduction,Kop89} ---in particular, Stone duality between $\BA$ and $\BS$ is \cite[Theorem 8.2]{Kop89}.

\subsection{Locally finite MV-algebras and multisets}\label{ss:lf-mv-algebras}

In \cite{CigDubMun04}, the authors give a description of $\mathsf{pro}\text{-}\MS_{\mathsf{f}}$  in terms of ``possibly infinite multisets''.  It turns out that infinity in $\mathsf{pro}\text{-}\MS_{\mathsf{f}}$ appears in two different aspects: infinite cardinality of the underlying set and infinitely-valued denominators.

\begin{definition}
A \emph{supernatural number} is a function
\[
\nu\colon \P\longrightarrow \N\cup\{\infty\}.
\]
Upon writing $\nu\leq \mu$ iff $\nu(p)\leq \mu(p)$ for each $p\in\P$, the set of supernatural numbers forms a complete lattice denoted by $\NN$.
We say that $\nu$ is \emph{finite} iff $\infty$ does not belong to the range of $\nu$ and $\nu(p)$ is nonzero only for finitely many $p$.
\end{definition}
Regarding $\nu$ as assigning exponents to prime numbers and using the Unique Factorisation Theorem, we may assign to each number $n\in\Nnot$ a corresponding finite supernatural number $\nu_n$. In particular, $\nu_1$ constantly takes value $0$ on every prime number.
The one-one correspondence $n\leftrightarrow \nu_n$ identifies the sub-lattice of finite supernatural numbers with the  lattice $(\Nnot, \mathrm{div})$, where $a \mathrel{\mathrm{div}} b$ if and only if $a$ divides $b$. 

We equip $\NN$ with the topology induced by the following open basis:
\begin{align}\label{e:basis-top-N} 
U_n\df \{\nu\in \NN\mid\nu\geq \nu_n\} \text{ for }n\in \Nnot
\end{align}
or equivalently by the open sub-basis of all sets of the form
\begin{align*}
U_{p,k} = \{\nu \in \NN \mid \nu(p) > k\}, \text{ for } p \in \P\text{ and } k \in \N.
\end{align*}
The definition of the basis differs from the one proposed in \cite{CigDubMun04} ---which probably contains a typo--- but agrees with the one in \cite[Section 8.4]{Mun2011}.  We postpone to the Appendix a discussion on the various possible formulations and their equivalences.

Supernatural numbers (sometimes called generalised natural numbers or Steinitz numbers) were introduced by E.\ Steinitz in 1910 under the name of ``$G$-Zahlen'' (where $G$ stands for ``Grad'') as a part of his study of the Galois group of field extensions of possibly infinite degree \cite[Section~16, p.\ 250]{Steinitz1910}.
The lattice of supernatural numbers can be regarded as the ideal-completion of the lattice $(\Nnot, \mathrm{div})$ of positive integers equipped with the divisibility order, where by ideal we mean here a up-directed downward-closed subset.

In the literature, supernatural numbers arise especially in the study of groups, since the divisibility relation between positive integers comes into play (cf., for example, Lagrange's theorem).
In the context of profinite groups (such as the Galois groups considered by Steinitz), supernatural numbers are used to define the order of a profinite group and the index of a closed subgroup of a profinite group, see \cite[Chapter~1, Section~2]{Shatz1972} or \cite[Chapter~1, Section~1.3]{Serre1994}.
The lattice of supernatural numbers is isomorphic to the lattice of additive subgroups of $\Q$ containing $1$, ordered by inclusion; for this reason, a quotient of the set of supernatural numbers is used to classify isomorphism classes of Abelian torsion-free groups of rank $1$, see \cite{Baer1937} or \cite[Chapter~XIII, Section~85]{Fuchs1973}.
Using similar ideas, J.\ G.\ Glimm showed in \cite[Theorem 1.12]{zbMATH03153447} that it is possible to associate to every uniformly hyperfinite C$^{*}$-algebra $A$ a supernatural number that completely characterises the isomorphism type of $A$.

\begin{definition}\label{d:multiset}
We define the category $\MS$ of \emph{multisets}.  Objects are pairs $(X,\zeta)$, where $X$ is a Boolean space, and $\zeta$ is a continuous map from $X$ to $\NN$, equipped with the topology defined in \eqref{e:basis-top-N}.
If $(X,\zeta_X)$ and  $(Y,\zeta_Y)$ are two multisets, an arrow between them is a continuous function $f\colon X\to Y$ such that, 
\begin{align}\label{eq:respect-denominators}
\text{for every  }x\in X,\quad  \zeta_X(x)\geq \zeta_Y(f(x)).
\end{align}
\end{definition}
When more than one multiset is involved in an argument we write $\zeta_X$ to make it clear that we refer to the map associated with the multiset $X$.
\begin{theorem}[\mbox{\cite[Theorem 6.8]{CigDubMun04}}]\label{t:duality}
The categories $\MVlf$ and $\MS$ are dually equivalent.
\end{theorem}
The duality of the preceding theorem was extended in \cite{Cignoli_2012} to encompass \emph{locally weakly finite} MV-algebras and \emph{real-valued} multisets.
\begin{remark}\label{r:duality}
The above duality is part of a larger duality between semisimple MV-algebras\footnote{Notice that every locally finite MV-algebra is semisimple. We warn the reader that C.\ C.\ Chang misleadingly used the term \emph{locally finite} as a synonym for \emph{semisimple}.} and closed subspaces of the hyper-cubes $[0,1]^{\kappa}$, for $\kappa$ that ranges among cardinals. See \cite{zbMATH06088860,zbMATH06137396} and in particular \cite[Remark 10]{zbMATH06688534} where the duality of \cref{t:duality} is framed in the setting of \cite{zbMATH06088860}.  Here we limit ourselves to notice that every multiset can be homeomorphically embedded into some power of $[0,1]$ in such a way that the denominator function $\zeta$ equals the least common multiple of the denominators of coordinates of the (necessarily rational) point in the image of the embedding.  For more information about this kind of embeddings and for a topological characterisation of the maps from a compact and Hausdorff space into $\N$ that are concretely representable as ``denominator maps'' we refer the reader to \cite{AbMaSpa20}.
\end{remark}

For the reasons sketched in \cref{r:duality}, if $(X,\zeta)$ is a multiset we call $\zeta(x)$ the \emph{denominator of $x$}, for every $x\in X$, and $\zeta$ is called \emph{denominator map}.
When there is no danger of confusion, we write $X$ instead of $(X,\zeta)$. Since $(\Nnot, \mathrm{div})$ embeds into $\NN$, we also say that a function $f \colon X \to Y$ \emph{decreases denominators} when it satisfy the property described in \eqref{eq:respect-denominators}. Further, we say that $f$ \emph{preserves denominators} whenever, for every $x \in X$, we have $\zeta_X(x) = \zeta_Y(f(x))$.

It should be noted that the requirement that the arrows in $\MS$ decrease denominators depart from several other generalisations of finite multisets to be found in  \cite{Blizard1989,Juergensen2020,Lake1975,Monro1987,Syropoulos2001}.  However, the category dual to $\MVlf$ ---which is of prime interest in this paper--- must contemplate this property as it directly reflects the behaviour of the arrows between finite MV-algebras, as explained in the Introduction.

\subsection{Properties of the category of multisets}
\label{ss:category-MS}
In this subsection, we investigate some basic categorical constructions in $\MS$.  The following crucial remark and the ensuing \cref{t:U-topological} greatly simplify calculations.

The category $\MS$ has a natural forgetful functor $U$ into the category $\BS$ of Boolean spaces which simply forgets denominator maps.
Vice versa, there are at least two natural denominator maps that can be attached to any Boolean space $X$ to make it into a multiset: the denominator map $\zeta_{1}$, where $\zeta_{1}(x)(p)\df0$ for every $x\in X$ and $p\in \P$, and the denominator map $\zeta_{\infty}$, where $\zeta_{\infty}(x)(p)\df\infty$ for every $x\in X$ and $p\in \P$. 
This relation between the two categories resembles the one between $\mathsf{Set}$ and $\mathsf{Top}$ in which one can endow any set with the discrete or the indiscrete topology.
In this subsection, we will see that this similarity can be formally stated by proving that the forgetful functor $U\colon \MS\to\BS$ is \emph{topological} (see \cref{t:U-topological} below).
From this fact, we will derive many consequences about limits, colimits, epic and monic arrows, some of which were already observed in \cite{CigDubMun04}.

We recall basic notions and results concerning topological functors. For more details, we refer to \cite[Chapter~21]{AdaEtal2006}.
Given a faithful functor $G\colon \mathsf{A}\to \mathsf{X}$, a family of arrows $\{f_i\colon A\to A_i\}_{i\in I}$ in $\mathsf{A}$ is called \emph{$G$-initial} provided that, for each arrow $h\colon G(B)\to G(A)$ in $\mathsf{X}$, if for every $i\in I$ there exists an arrow $g_{i}\colon B\to A_{i}$ in $\mathsf{A}$ with $G(g_{i})=G(f_i)\circ h$, then there exists an arrow $h_{0}\colon B\to A$ in $\mathsf{A}$ such that $G(h_{0})=h$ (in loose terms, $h$ is an $\mathsf{A}$-arrow whenever all compositions $G(f_i)\circ h$ are so).
\[\begin{tikzcd}
B \arrow[swap, dashed]{d}{\exists!h_{0}} \arrow{rd}{g_i} & \\
A \arrow[swap]{r}{f_i} & A_i
\end{tikzcd}
\quad \quad 
\begin{tikzcd}
G(B) \arrow[swap]{d}{G(h_{0})=h} \arrow{rd}{G(g_i)} & \\
G(A) \arrow[swap]{r}{G(f_i)} & G(A_i)
\end{tikzcd}
\]
Furthermore, we say that a family of arrows $\{\overline{f}_i\colon A\to A_i\}_{i\in I}$ in $\mathsf{A}$ is 
a \emph{lift} of $\{f_i\colon X\to G(A_i)\}_{i\in I}$ if $G(A)=X$ and $G(\overline{f}_i)=f_i$. Finally, a faithful functor $G\colon \mathsf{A}\to \mathsf{X}$ is called \emph{topological} (see e.g.\ \cite[Definition 21.1]{AdaEtal2006}) provided that every class-indexed family of arrows $\{f_i\colon X\to GA_i\}_{i\in I}$ in $\mathsf{X}$ has a unique $G$-initial lift.

\begin{lemma}\label{l:join-is-continuous}
Let $X$ be a topological space, and let $\{f_i\colon X\to \NN\}_{i\in I}$ be a family of continuous functions.
Then the function $\bigvee_{i \in I} f_i \colon X \to \NN$ that maps $x$ to $\bigvee_{i\in I} f_i(x)$ is continuous.
\end{lemma}
\begin{proof}
A proof, which can also be found in \cite[Lemma 3.4]{CigDubMun04}, runs as follows.
For $p\in\P$ and $k\in \N$ we have
\begin{align*}
\left(\bigvee_{i \in I} f_i\right)^{-1} [U_{p,k}]& = \left\{x \in X \mid \bigvee_{i \in I} f_i(x) \in U_{p,k}\right\} \\
& = \left\{x \in X \mid \bigvee_{i \in I} f_i(x)(p) > k\right\} \\
& = \{x \in X \mid \exists i \in I \text{ s.t.\ } f_i(x)(p) > k\}\\
& = \bigcup_{i \in I} \{x \in X \mid f_i(x)(p) > k\} \\
& = \bigcup_{i \in I} f_i^{-1} [U_{p,k}].
\end{align*}
\end{proof}

We let $U$ denote the forgetful functor from $\MS$ to $\BS$.
\begin{theorem}\label{t:U-topological}
Every family of arrows $\{f_i\colon X\to U(X_i,\zeta_{X_i})\}_{i\in I}$  in $\BS$ admits a unique $U$-initial lift, namely $\{\bar{f}_i\colon (X, \zeta_X) \to (X_i,\zeta_{X_i})\}_{i\in I}$, where, for each $x \in X$,
\[
\zeta_X(x)\coloneqq \bigvee_{i\in I}\zeta_{X_i}(f_i(x)),
\]
and, for each $i \in I$, $\bar{f}_i$ is just the function $f_i$.
Therefore the functor $U\colon \MS\to \BS$ is topological.
\end{theorem}
\begin{proof}
By \cref{l:join-is-continuous}, $(X,\zeta_X)$ is a multiset.
Notice that each $f_i\colon X\to X_i$ decreases denominators with respect to $\zeta_{X_{i}}$ and the newly defined $\zeta_{X}$, thus it corresponds to an arrow of multisets, denoted by $\bar{f}_{i}$ and $U(\bar{f}_{i})=f_{i}$.
Hence the family $\{f_i\colon X\to U(X_i,\zeta_{X_i})\}_{i\in I}$ has a lift $\{\bar{f}_{i}\colon (X,\zeta_X)\to (X_i,\zeta_{X_i})\}_{i\in I}$.
Let us prove that it is $U$-initial.
Let $(C,\zeta_C)$ be a multiset and  $h\colon C\to X$ be an arrow in $\BS$. Suppose that each composite $U(\bar{f}_{i})\circ h\colon C\to X_i$ decreases the denominators, i.e., for every $i\in I$, $\zeta_{X_i}(U(\bar{f}_{i})(h(c)))\leq \zeta_C(c)$.
Therefore, $\zeta_X(h(c))=\bigvee_{i\in I}\zeta_{X_i}(f_{i}(h(c)))\leq \zeta_C(c)$.
This shows that $h$ decreases  denominators, so $\{\bar{f}_i\colon (X,\zeta_X)\to (X_i,\zeta_{X_i})\}_{i\in I}$ is $U$-initial.
Uniqueness follows from the easily verifiable fact that the functor $U\colon \MS\to \BS$ has the property that an iso $f$ in $\MS$ is an identity whenever $U(f)$ is an identity, see \cite[Proposition 21.5]{AdaEtal2006}\footnote{A functor with this property is called amnestic, see \cite[Definition 3.27]{AdaEtal2006}.}.
\end{proof}
\begin{corollary}\label{c:initial}
A  family of arrows $\{f_i\colon (X,\zeta)\to (X_i,\zeta_{X_i})\}_{i\in I}$ in $\MS$ is $U$-initial if and only if, for every $x\in X$, we have $\zeta_X(x)=\bigvee_{i\in I}\zeta_{X_i}(f_i(x))$.
In particular, an arrow of multisets $f\colon (X,\zeta_X)\to (Y,\zeta_Y)$ is $U$-initial if and only if it preserves denominators.
\end{corollary}
\begin{proof}
A family of arrows $\{f_i\colon (X,\zeta)\to (X_i,\zeta_{X_i})\}_{i\in I}$ in $\MS$ is clearly a lift of $\{U(f_i)\colon X \to U(X_i,\zeta_{X_i})\}_{i\in I}$.
By the description and the uniqueness of $U$-initial lifts (\cref{t:U-topological}), $\{f_i\colon (X,\zeta)\to (X_i,\zeta_{X_i})\}_{i\in I}$ is $U$-initial if and only if $\zeta_X(x)=\bigvee_{i\in I}\zeta_{X_i}(x)$.
\end{proof}
\begin{corollary}[\mbox{Cf.\ \cite[Theorem 8.10]{zbMATH01382786} and \cite[Corollary 6.12]{CigDubMun04}}]\label{c:complete-and-co}
The category of multisets is complete and co-complete.
\end{corollary}
\begin{proof}
Since $\BS$ is dually equivalent to the variety of Boolean algebras, is complete and co-complete.
By \cite[Theorem 21.16.(1)]{AdaEtal2006}, if $G\colon \mathsf{A}\to \mathsf{X}$ is topological, then $\mathsf{A}$ is (co-)complete if and only if $\mathsf{X}$ is (co-)complete.
The claim then follows from \cref{t:U-topological}.
\end{proof}
\begin{remark} \label{r:adjoints}
Every topological functor $G \colon \mathsf{A} \to \mathsf{X}$ has a left adjoint $L'$ (the discrete functor) and a right adjoint $R'$ (the indiscrete functor), which are full embeddings satisfying $G \circ L' = 1_\mathsf{X}$ and $G \circ R' = 1_{\mathsf{X}}$ \cite[Proposition 21.12]{AdaEtal2006}.
The left adjoint $L$ to the forgetful functor $U\colon \MS \to \BS$ maps a Boolean space $S$ to the multiset $(S,\zeta_{\infty})$, where $\zeta_{\infty}(x)(p)=\infty$ for every $x\in S$ and $p\in \P$.
The right adjoint $R$ to $U$ maps a Boolean space $S$ to the multiset $(S,\zeta_{1})$, where $\zeta_{1}(x)(p)=0$ for every $x\in S$ and $p\in \P$. 
The functor $R \colon \BS \to \MS$ is a full embedding, and $U$ is left adjoint to $R$; it follows that $\BS$ is reflective in $\MS$.
For similar reasons, $\BS$ is co-reflective in $\MS$.
Since the forgetful functor $U \colon \MS \to \BS$ has both a right and left adjoint it preserves both limits and co-limits.  This entails that limits in the category $\MS$ can easily be described in terms of limits in the category of Boolean spaces.
\end{remark}
\begin{lemma} \label{l:limits}
Let $D\colon \mathsf{I}\to \MS$ be a diagram.
The family $\mathcal{L}=\{l_i\colon (X, \zeta_X)\to D(i)\}_{i\in \mathsf{I}}$ is a limit of $D$ if and only if $U(\mathcal{L})$ is a limit of $U \circ D$ and $\zeta_X(x) = \bigvee_{i\in \mathsf{I}}\zeta_{D(i)}(l_i(x))$, for every $x \in X$.
\end{lemma}
\begin{proof}
The forgetful functor $U \colon \MS \to \BS$ is faithful and preserves all limits by \cref{r:adjoints}.
Therefore, $\mathcal{L}$ is a limit of $D$ if and only if $U(\mathcal{L})$ is a limit of $U \circ D$ and $\mathcal{L}$ is $U$-initial \cite[Proposition 13.15]{AdaEtal2006}.
By \cref{c:initial}, $\mathcal{L}$ is $U$-initial if and only if $\zeta_X(x) = \bigvee_{i\in \mathsf{I}}\zeta_{D(i)}(l_i(x))$, for every $x\in X$.
\end{proof}
\begin{remark} \label{r:reflective}
We recall that limits in $\BS$ are the same as in $\Top$ (the category of topological spaces and continuous maps).
This is guaranteed by the following facts.
The category $\BS$ is a reflective full subcategory of the category of compact Hausdorff spaces and continuous maps: the reflector assigns to each compact Hausdorff space the space of its connected components \cite[Proposition 5.7.12]{BJ01}.
Furthermore, the category of compact Hausdorff spaces is a reflective full subcategory of the category $\Top$: the reflector assigns to each topological space its Stone-\v{C}ech compactification (see \cite[Chapter IV, Section 2]{Joh86}).
Hence, the category $\BS$ is a reflective full subcategory of $\Top$.
Therefore, the forgetful functor from $\BS$ to $\Top$ preserves and reflects limits.
\end{remark}
A characterisation of $U$-final lifts (i.e.\ the dual notion of $U$-initial lifts) and co-limits in $\MS$ in terms of co-limits in $\BS$ is also available, although, in general, it is not as explicit as for $U$-initial lifts and limits.

\begin{lemma}\label{l:final}
Every family of arrows $\{f_i\colon U(X_i,\zeta_{X_i}) \to X\}_{i\in I}$ in $\BS$ admits a unique $U$-final lift, which is $\{\bar{f}_i \colon (X_i,\zeta_{X_i}) \to (X, \zeta_X)\}_{i\in I}$, where $\zeta_X$ is the greatest (with respect to the pointwise order) continuous function $\zeta \colon X \to \NN$ such that, for every $i \in I$ and $x \in X_i$, we have $\zeta(l_i(x)) \leq \zeta_{X_i}(x)$.
If in addition $I$ is finite, then
 \[\zeta_X(x) = \bigwedge_{i \in I, y \in X_i: l_i(y) = x} \zeta_{X_i}(y).\]
\end{lemma}
\begin{proof}
By \cref{t:U-topological}, every family of arrows $\{l_i\colon U(X_i,\zeta_{X_i})\to X\}_{i\in I}$ in $\BS$ admits a unique $U$-final lift  ---see the Topological Duality Theorem \cite[Proposition 21.9]{AdaEtal2006}.
Using the definition of $U$-finality, one obtains that the denominator $\zeta_X$ in the $U$-final lift is the greatest (with respect to the pointwise order) continuous function $\zeta \colon X \to \NN$ such that, for every $i \in I$, the function $f_i \colon X_i \to X$ is the underlying function of an arrow $(X_i,\zeta_{X_i})\to (X, \zeta)$. In other words, $f_i$ preserves denominators, i.e., for every $x \in X_i$, $\zeta(l_i(x)) \leq \zeta_{X_i}(x)$.
(Note that the existence of such a function is guaranteed by the existence of a $U$-final lift, together with the fact that a $U$-final lift must satisfy this property.)

We prove that, when $I$ is a finite set, the function $\zeta_X$ is precisely
\begin{equation}\label{e:inf-for-colimits}
\zeta_X(x) = \bigwedge_{i \in I, y \in X_i: l_i(y) = x} \zeta_{X_i}(y).
\end{equation}
Verifying this statement amounts to proving that the so-defined $\zeta_X$ is continuous, the remaining conditions are easily seen to hold.
We prove that the preimage under $\zeta_X$ of a closed set is closed: for every $p \in \P$ and $k \in \N$ we have
\begin{equation} \label{e:preimage}
\begin{split}
\zeta_{X}^{-1}[\NN \setminus U_{p,k}] & = \zeta_X^{-1}[\{\nu \in \NN \mid \nu(p) \leq k\}] \\
& = \{x \in X \mid \zeta_X(x) \in \{\nu \in \NN \mid \nu(p) \leq k\}\} \\
& = \{x \in X \mid \zeta_X(x)(p) \leq k\} \\
& = \left\{x \in X \mid \left(\bigwedge_{i \in I, y \in X_i: l_i(y) = x} \zeta_{X_i}(y)\right)(p) \leq k\right\}\\
& = \left\{x \in X \mid \bigwedge_{i \in I, y \in X_i: l_i(y) = x} \zeta_{X_i}(y)(p) \leq k\right\} \\
& = \left\{x \in X \mid \exists i \in I, \exists y \in X_i: l_i(y) = x, \zeta_{X_i}(y)(p) \leq k\right\} \\
& = \bigcup_{i \in I} \left\{x \in X \mid \exists y \in X_i: l_i(y) = x, \zeta_{X_i}(y)(p) \leq k\right\} \\
& = \bigcup_{i \in I} \left\{x \in X \mid \exists y \in \zeta_{X_i}^{-1}[\NN \setminus U_{p,k}]: l_i(y) = x\right\} \\
& = \bigcup_{i \in I} l_i\big[\zeta_{X_i}^{-1}[\NN \setminus U_{p,k}]\big],
\end{split}
\end{equation}
which is closed since $I$ is finite.
\end{proof}

We remark that when $I$ is infinite \eqref{e:inf-for-colimits} may fail to define the denominator map of the $U$-final lift.
\begin{lemma}\label{l:colimits-MS}
	Let $D \colon \mathsf{I}\to \MS$ be a diagram.
	The family $\mathcal{L} = \{l_i \colon D(i) \to (X, \zeta_X)\}_{i \in \mathsf{I}}$ is a co-limit of $D$ if and only if $U(\mathcal{L})$ is a co-limit of $U \circ D$ and $\zeta_X$ is the greatest continuous function $\zeta \colon X \to \NN$ such that, for every $i \in I$ and $x \in X_i$, $\zeta(l_i(x)) \leq \zeta_{D(i)}(x)$.
	If $I$ is finite, then 
	\[\zeta_X(x) = \bigwedge_{i \in I, y \in D(i): l_i(y) = x} \zeta_{D(i)}(y).\]
\end{lemma}
\begin{proof}
The forgetful functor $U \colon \MS \to \BS$ is faithful and preserves all co-limits by \cref{r:adjoints}.
Therefore, $\mathcal{L}$ is a co-limit of $D$ if and only if $U(\mathcal{L})$ is a co-limit of $U \circ D$ and $\mathcal{L}$ is $U$-final \cite[Proposition~13.15]{AdaEtal2006}.
Using the characterisation of $U$-final lifts available from \cref{l:final}, we have the needed result.
\end{proof}

By the previous lemma, co-limits in $\MS$ are built on co-limits in $\BS$.
Moreover, this construction is most explicit in the case of finite co-limits.
Hence, it might be useful to recall the following property of finite co-limits in $\BS$.

\begin{lemma}\label{l:reflect-colimits}
The forgetful functor from $\BS$ to $\Set$ reflects finite co-limits.
\end{lemma}
\begin{proof}
Let us denote with $\lvert - \rvert$ the forgetful functor from $\BS$ to $\Set$.
Let $D\colon \mathsf{I}\to \BS$ be a finite diagram, and let $\mathcal{L} = \{l_i\colon  D(i) \to X\}_{i\in \mathsf{I}}$ be a co-cone in $\BS$ such that $\{\lvert l_i\rvert \colon \lvert D(i) \rvert \to \lvert X \rvert\}_{i\in \mathsf{I}}$ is a co-limit in $\Set$.
We claim that a subset $Z$ of $X$ is closed if and only if, for every $i \in \mathsf{I}$, the set $l_i^{-1}[Z]$ is a closed subset of $D(i)$.
The left-to-right implication follows from continuity of $l_i$ for each $i \in \mathsf{I}$.
Let us prove the converse direction.
Suppose that, for every $i \in \mathsf{I}$, the set $l_i^{-1}[Z]$ is closed. Recall the Closed Map Lemma: a continuous function from a compact space to a Hausdorff space is closed.
It follows that the set $l_i[l_i^{-1}[Z]] = Z \cap l_i[D(i)]$ is closed.
Since $\{\lvert l_i\rvert \colon \lvert D(i) \rvert \to \lvert(X, \zeta_X)\rvert\}_{i\in \mathsf{I}}$ is a co-limit in $\Set$, we have $X = \bigcup_{i \in \mathsf{I}} l_i[D(i)]$.
Therefore, 
\[
	Z = Z \cap X = Z \cap \bigcup_{i \in \mathsf{I}} l_i[D(i)] = \bigcup_{i \in \mathsf{I}} Z \cap l_i[D(i)].
\]
This set is closed because it is a union of finitely many closed sets and this proves our claim.
Since $\{\lvert l_i\rvert \colon \lvert D(i) \rvert \to \lvert X \rvert\}_{i\in \mathsf{I}}$ is a co-limit in $\Set$ and by the previous claim the topology on $X$ is final, the diagram $\{l_i \colon D(i) \to X\}_{i\in \mathsf{I}}$ is a co-limit in $\Top$.  Since $\BS$ fully embeds in $\Top$, the latter is also a co-limit in $\BS$.
\end{proof}
The fact that $U$ is topological helps to characterise some general categorical concepts in $\MS$, e.g., regular monic and regular epic. These concepts will play important r{\^o}les in the rest of the paper.

Recall (see, e.g., \cite[Chapter 0]{AlgThe} for this and the subsequent notions) that an arrow $m$ is \emph{extremal monic} if it is monic and whenever $m=g\circ e$ with $e$ epic, then $e$ is iso.  The dual concept defines \emph{extremal epic} arrows.
Also recall that an arrow $m\colon A\to B$ is called \emph{regular monic} if there exists a pair of parallel arrows $f, g \colon B\to C$ for which $m$ is an equaliser, i.e., $f\circ m=g\circ m$ and for every arrow $n\colon D\to B$ with the same property there exists a unique arrow $u$ such that $n=m\circ u$.
Dually, an arrow $m\colon B\to A$ is called \emph{regular epic} if it is a co-equaliser of a pair of parallel arrows $f,g\colon C\rightrightarrows B$.
\begin{lemma} \label{l:arrow-in-Stone}
The following equivalences hold for any arrow in $\BS$:
\begin{enumerate}
\item \label{i:st-mono} regular monic $\Leftrightarrow$ extremal monic $\Leftrightarrow$ monic $\Leftrightarrow$ injective;
\item \label{i:st-epi} regular epic $\Leftrightarrow$ extremal epic $\Leftrightarrow$ epic $\Leftrightarrow$ surjective;
\item \label{i:st-iso} iso $\Leftrightarrow$ bijective.
\end{enumerate}
\end{lemma}
\begin{proof}
We start with some general considerations. Recall  (see, e.g., \cite[0.17]{AlgThe}) that in every category any regular monic arrow is extremal monic and every extremal monic is monic. Similarly, in every category, any regular epic arrow is extremal epic and extremal epic arrows are epic.
Furthermore, since the category $\BA$ of Boolean algebras with homomorphisms is a variety of algebras, in $\BA$ monic arrows are precisely the injective homomorphisms, and regular epic arrows are precisely the surjective homomorphisms. 
Applying Stone duality between $\BS$ and $\BA$, we obtain that in $\BS$ every epic arrow is surjective and every injective arrow is regular monic.

Now, regarding \cref{i:st-mono} it is easy to see that in $\BS$ every monic arrow is injective, so the proof of \cref{i:st-mono}  is settled.

To prove \cref{i:st-epi}, recall from \cref{r:reflective} that $\BS$ is a reflective full subcategory of $\Top$.
Since the inclusion functor of any reflective full subcategory reflects regular epimorphisms (see \cite[Exercise 7F.(c)]{AdaEtal2006} for a list of properties satisfied by the inclusion functor of reflective full subcategories), the forgetful functor from $\BS$ to $\Top$ reflects regular epimorphisms.
Recall that the regular epimorphisms in $\Top$ are precisely the topological quotient maps \cite[Examples 7.72.(2)]{AdaEtal2006}.
We deduce that a topological quotient map between Boolean spaces is a regular epic arrow in $\BS$.
It follows from the Closed Map Lemma that any surjective continuous map between Boolean spaces is a topological quotient map.
Therefore, every surjective arrow is regular epic in $\BS$ and \cref{i:st-epi} is settled.

Finally, an arrow is iso if and only if it is extremal monic and epic \cite[Proposition~7.66]{AdaEtal2006}; \cref{i:st-iso} follows.
\end{proof}

\begin{lemma}\label{l:arrows-in-MS} 
Let $f\colon X\to Y$ be an arrow in $\MS$.
\begin{enumerate}
\item\label{l:mono} $f$ is monic $\Leftrightarrow$ $f$ is injective.
\item\label{l:rmono} $f$ is regular monic $\Leftrightarrow$ $f$ is extremal monic $\Leftrightarrow$ $f$ is injective and preserves denominators.
\item\label{l:epi} $f$ is epic $\Leftrightarrow$ $f$ is surjective.
\item\label{l:repi} $f$ is regular epic $\Leftrightarrow$ $f$ is extremal epic $\Leftrightarrow$ $f$ is surjective and, for every $y \in Y$, we have $\zeta_Y(y) = \bigwedge_{x \in X : f(x) = y} \zeta_X(x)$.
\item\label{l:isom} $f$ is iso $\Leftrightarrow$ $f$ is bijective and preserves denominators.
\end{enumerate}
\end{lemma}
\begin{proof}
\Cref{l:epi,l:mono} are immediate consequences of the facts that topological functors preserve and reflect monic and epic arrows \cite[Proposition 21.13.(1)]{AdaEtal2006} and that, by \cref{l:arrow-in-Stone}, in $\BS$ monic arrows are precisely the injective ones and epic arrows are precisely the surjective ones.
Furthermore, if $G\colon \mathsf{A}\to \mathsf{X}$ is topological, then an arrow in $\mathsf{A}$ is an extremal (resp.\ regular) monic arrow if and only if it is $G$-initial and an extremal (resp.\ regular) monic arrow in $\mathsf{X}$ \cite[Proposition~21.13.(2)]{AdaEtal2006}.
By \cref{c:initial}, the arrow $f$ is initial if and only if it preserves denominators.
\Cref{l:rmono} follows.
Similarly, for a topological functor $G\colon \mathsf{A}\to \mathsf{X}$, an arrow in $\mathsf{A}$ is an extremal (resp.\ regular) epic arrow if and only if it is $G$-final and an extremal (resp.\ regular) epic arrow in $\mathsf{X}$ \cite[Proposition 21.13.(3)]{AdaEtal2006}.
By \cref{l:final}, an arrow $f$ in $\MS$ is $U$-final if and only if we have $\zeta_Y(y) = \bigwedge_{x \in X : f(x) = y} \zeta_X(x)$.
\Cref{l:repi} follows.
Finally, since an arrow is iso if and only if it is extremal monic and epic, we obtain \cref{l:isom}.
\end{proof}
\begin{remark}
Given a multiset $X$, a regular monic arrow in $\MS$ with codomain $X$ is, up to an iso, a closed subspace of $X$ with the induced denominator map.
\end{remark}

\subsection{Algebraic categories}
\label{ss:algebraic-categories}

In this subsection, we recall a well-known characterisation of those categories which are equivalent to some (quasi-)variety of algebras  (\Cref{t:char-quasi-varieties} below).

As customary in the study of algebraic category, see e.g.\ \cite{Adamek2004}, we admit \emph{infinitary} algebras, i.e.\ algebras with operations of infinite arity, and we allow \emph{large} signatures (i.e.\ a class, rather than a set).
In particular, we work with a large signature $\Sigma$ which is the union of the classes of $\kappa$-ary operations, for $\kappa$ cardinal.

A \emph{quasi-equation} is a (universally quantified) formula
\[ 
\left(\bigwedge_{i\in I}(u_i=v_i)\right)\Longrightarrow (u_0=v_0),
\]
where $I$ is a (possibly infinite) set, and $u_i$, $v_i$ are, for $i\in I\cup \{0\}$, terms over a given set of variables.

Following \cite{Adamek2004}, a class $\mathsf{A}$ of $\Sigma$-algebras is called a \emph{quasi-variety of $\Sigma$-algebras} (resp.\ \emph{a variety of algebras}) if
\begin{enumerate}
	\item \label{i:presented-by} the class $\mathsf{A}$ can be presented by a class of quasi-equations (resp.\ equations), and
	\item \label{i:free-algs} the class $\mathsf{A}$ has free algebras (equivalently, for each cardinal $\kappa$, the class $\mathsf{A}$ has only a set of isomorphism classes of algebras on $\kappa$ generators).
\end{enumerate}

We now provide the background needed to state the characterisation of varieties and quasi-varieties of algebras.
\begin{definition}\label{d:generator}
A set of objects $\GG$ is \emph{generating} provided that for each pair $f_1,f_2\colon K\rightrightarrows K'$ of distinct parallel arrows there exists an object $G\in \GG$ and an arrow $g\colon G\to K$ such that $f_1\circ g\neq f_2\circ g$ \cite[0.6]{AdaRos1994}.

In a category with co-products, a set of objects $\GG$ is generating if and only if, for every object $A$, the canonical arrow, obtained via an application of the universal property of co-products to the co-cone $\{h \colon G \to A\}_{G \in \GG,\, h \in \hom(G, A)}$,
\begin{equation}\label{eq:generator}
\sum_{G\in \GG,\,h \in \hom(G, A)}{G} \to A
\end{equation}
is epic.
\end{definition}
\begin{definition}\label{d:regular-generator}
A set of objects $\GG$ is \emph{regularly generating} if the hom-functors $\hom(G,-)$ (for $G \in \GG$) collectively reflect regular arrows.
In categories with co-products, this is equivalent (see \cite[Section 5.1]{Adamek2004}) to the fact that the canonical quotient in \eqref{eq:generator}  is regular epic.
As a special case, we have that an object $G$ is a \emph{regular generator} if the hom-functor $\hom(G,-)$ reflects regular epic arrows.
As observed in \cite[Section 1.1]{Adamek2004}, if the object $G$ has co-powers, this is equivalent to the following condition:
for every object $A$, the canonical arrow
\[
\sum_{\hom(G,A)}{G}\to A
\]
is regular epic.
\end{definition}
\begin{definition}
Recall that an object $P$ is called \emph{regular projective} if the hom-functor $\hom(P,-)$ preserves regular epics.
In other words, for any arrow $f\colon P\to B$ and every regular epic arrow $g\colon A\repi B$, the arrow $f$ factors through $g$, i.e., there exists $h$ such that the following diagram commutes:
\[
\begin{tikzpicture}

\node(A) at (1.5,1.5)   {$A$};
\node (B) at (1.5,0) {$B$};
\node (P) at (0,0) {$P$};

\draw [- open triangle 60] (A) -- (B) node [right, midway] {$g$};
\draw [->] (P) -- (B) node [below, midway] {$f$};
\draw [->, dashed] (P) -- (A) node [above, midway] {$h$};

\end{tikzpicture}
\]
\end{definition}
\begin{definition}
A set of objects $\GG$ is \emph{abstractly finite} if every arrow from an object of $\GG$ to a co-product of objects in $\GG$ factors through a finite sub-co-product.
As a particular case, an object $G$ is called \emph{abstractly finite} if every arrow from $G$ to a co-power of $G$ factors through a finite sub-co-power.
\end{definition}
\begin{definition}\label{d:finitely-presentable}
Following \cite[6.1]{MR0327863}, we say that an object $A$ in a category $\C$ is \emph{finitely presentable} if the covariant hom-functor $\hom{(A,-)}\colon$ $\C \to \Set$ preserves filtered co-limits.
Explicitly, this means that if $\mathsf{I}$ is a filtered category and $D\colon \mathsf{I}\to \C$ is a functor with co-limit co-cone $\{b_{i} \colon D(i) \to B\}_{i \in \mathsf{I}}$,  then for every arrow $f \colon A \to B$ in $\C$ the following two conditions are satisfied:
\begin{enumerate}
\item[\namedlabel{e:factorization}{(F)}] There is $g \colon A \to D(i)$  such that $f = b_{i}\circ g$.
\item[\namedlabel{e:essential-unique}{(E)}] For any $g',g'' \colon A \to D(j)$  such that $f = b_{j}\circ g' = b_{j}\circ g''$,  there is  $d_{jk} \colon j \to k$ such that  $D(d_{ij}) \circ g' = D(d_{ij})\circ g''$. 
\end{enumerate}
\begin{center}
\begin{tikzpicture}[descr/.style={fill=white}]
\matrix(m)[matrix of math nodes, row sep=4em, column sep=4em, text height=1.5ex, text depth=0.25ex, ampersand replacement=\&]{
A\& B\\\& D(k)\\\&D(j)\\
};
\path[<-, dashed] (m-2-2) edge node[right] {$D(d_{ij})$} (m-3-2);
\path[->] (m-1-1) edge node[above] {$f$} (m-1-2);
\path[->] (m-1-1) edge  [bend right=50] node[left] {$g'$} (m-3-2);
\path[->] (m-1-1) edge  [bend right=30] node[right] {$g''$} (m-3-2);
\path[->, dashed] (m-1-1) edge    [bend right=20] (m-2-2);
\path[->] (m-2-2) edge [bend right=60] node[right] {$b_k$} (m-1-2);
\path[->] (m-3-2) edge [bend right=90] node[right] {$b_j$} (m-1-2);
\end{tikzpicture}
\end{center}
Similarly, after \cite[6.1]{MR0327863}, we say that $A$  is \emph{finitely generated} if  $\hom{(A,-)}\colon \C \to \Set$ preserves filtered co-limits of diagrams $(P_{i}, p_{ij})$ all of whose transition arrows $p_{ij}$ are monic in $\C$.
\end{definition}
\begin{lemma}[\mbox{\cite[Proposition 2.4]{pedicchio2000abstract}}] \label{l:abstractly-finite-vs-finitely-generated}
If $G$ is an abstractly finite, regular projective, regular generator, then $G$ is finitely generated.
Vice versa, if $G$ is finitely generated and has co-powers, then $G$ is abstractly finite.
\end{lemma}
\begin{definition}[See, e.g., \mbox{\cite[Definition 3.12]{AlgThe}}]\label{d:equivalence-relation}
Let $\C$ be a category with finite limits and $A$ an object of $\C$.
An \emph{(internal) equivalence relation} on $A$ is a subobject 
$\langle p_0,p_1\rangle\colon R\mono A\times A$ satisfying the following properties:
\begin{description}
\item[reflexivity] there exists an arrow $d\colon A\to R$ in $\C$ such that the following diagram commutes:
\[
\begin{tikzcd}
A \arrow[swap,tail]{rd}{\langle 1_A,1_A\rangle} \arrow[dashed]{rr}{\exists d}&& R \arrow[tail]{dl}{\langle p_0,p_1\rangle}\\
& A \times A&
\end{tikzcd}
\]
\item[symmetry] there exists an arrow $s\colon R\to R$ in $\C$ such that the following diagram commutes:
\[
\begin{tikzcd}
R \arrow[dashed]{rr}{\exists s} \arrow[tail,swap]{dr}{\langle p_1,p_0\rangle}&& R \arrow[tail]{dl}{\langle p_0,p_1\rangle}\\
& A \times A&
\end{tikzcd}
\]
\item[transitivity] if the left-hand diagram below is a pullback square in $\C$, then there is an arrow $t\colon P\to R$ such that the right-hand diagram commutes.
\[
\begin{tikzcd}
P \arrow{r}{\pi_1} \arrow[swap]{d}{\pi_0} \arrow[dr, phantom, "\lrcorner", very near start]& R \arrow{d}{p_0} \\
R \arrow[swap]{r}{p_1} & A
\end{tikzcd}
\quad \quad
\begin{tikzcd}
P \arrow[swap]{rd}{\langle p_0\circ \pi_0,p_1\circ\pi_1\rangle} \arrow[dashed]{rr}{\exists t}&& R \arrow[tail]{dl}{\langle p_0,p_1\rangle}\\
& A \times A&
\end{tikzcd}
\]
\end{description}
\end{definition}
\begin{definition}[See, e.g., \mbox{\cite[0.15 and Definition 3.14]{AlgThe}}]\label{d:effective-exact}
Recall that the \emph{kernel pair} of an arrow $f\colon X\to Y$ is the pair of arrows $R \rightrightarrows X$ in the pullback of $f$ along itself:
\[
\begin{tikzcd}
R \arrow{r}\arrow{d} \arrow[dr, phantom, "\lrcorner" , very near start, color=black]& X\arrow{d}{f}\\
X \arrow[swap]{r}{f}& Y.
\end{tikzcd}
\]
It is folklore that every kernel pair is an equivalence relation.
An equivalence relation $\langle p_0,p_1\rangle\colon R\mono A\times A$ is \emph{effective} if there exists an arrow $q\colon A\to S$ such that $p_0,p_1\colon R\rightrightarrows A$ is the kernel pair of $q$.
\end{definition}
For varieties and quasi-varieties of algebras, the definition of equivalence relation given above coincides with the usual notion of congruence, while the effective equivalence relations in quasi-varieties are the so-called \emph{relative} congruences.

We have finally collected all necessary background to recall some well-known characterisations of varieties and other classes of algebras.
\begin{theorem}\label{t:char-quasi-varieties}
Let $\C$ be a (locally small) category.
\begin{enumerate}

\item \label{t:quasi-variety}
$\C$ is equivalent to a quasi-variety of  algebras if and only if
\begin{enumerate}
\item $\C$ is co-complete, and
\item $\C$ has a regular projective regular generator.
\end{enumerate}
\item \label{t:quasi-variety-finitary}
$\C$ is equivalent to a quasi-variety of finitary algebras if and only if
\begin{enumerate}
\item $\C$ is co-complete, and
\item $\C$ has an abstractly finite, regular projective regular generator.
\end{enumerate}
\item \label{t:quasi-variety-many-sorted}
$\C$ is equivalent to a many-sorted quasi-variety of finitary algebras if and only if
\begin{enumerate}
\item $\C$ is co-complete, and
\item $\C$ has an abstractly finite, regularly generating set of regular projective objects.
\end{enumerate}
\item \label{t:quasi-variety-effective}
All items above remain true if we replace ``quasi-variety'' with ``variety'' on the left side of the equivalence and add the condition ``(c) every equivalence relation in $\C$ is effective'' on the right one.
\end{enumerate}
\end{theorem}
\begin{proof}
\Cref{t:quasi-variety-finitary}, with the additional requirement that $\C$ has equalisers, was proved by Isbell in \cite[Theorem 5.2]{isbell1964subobjects}.
However, Ad\'{a}mek \cite[Theorem 3.6]{Adamek2004} proved that such an assumption can be dropped (and proved a characterisation stronger than the one presented here).
\Cref{t:quasi-variety,t:quasi-variety-many-sorted} can also be found in \cite[Characterisation Theorem 1 and Theorem 5.2]{Adamek2004}.
Finally, \cref{t:quasi-variety-effective} holds because in every variety equivalence relations are effective and, vice versa,  if in a quasi-variety equivalence relations are effective, then it is a variety (see e.g., \cite[Corollary 3.25]{AdaRos1994}).
\end{proof}
\begin{remark}\label{r:card-sorts-arity}
In the proof of \cref{t:quasi-variety-many-sorted} in \cref{t:char-quasi-varieties}, the number of objects in the regularly generating set corresponds to the number of sorts.
In the proof of \cref{t:quasi-variety} in \cref{t:char-quasi-varieties}, every operation in the variety depends on at most countably many coordinates if and only if every arrow from the regular projective regular generator $G$ to a co-power of $G$ factors through an at most countable sub-co-power.
\end{remark}

\section{The case of a finite language}
\label{s:finite-language}

In this section we prove that the category $\MVlf$ of locally finite MV-algebras is not equivalent to any quasi-variety in a language with finitary operations and finitely many sorts; the other weaker results stated in the introduction will follow at once. To achieve this we will prove:
\begin{enumerate}
	\item \label{i:fin-co-gen-implies-fin-more-sorts}
		every finitely co-generated object in $\MS$ has a finite underlying set;
	\item \label{i:reg-co-gen-implies-infin-more-sorts:2}
		every regularly co-generating set $\GG$ contains either infinitely many objects or an object whose underlying set is infinite.
\end{enumerate}
\Cref{i:fin-co-gen-implies-fin-more-sorts,i:reg-co-gen-implies-infin-more-sorts:2} together imply that $\MS$ admits no regularly co-generating finite set of finitely co-generated objects,  and thus the dual statement holds for $\MVlf$.  An application of \cref{t:char-quasi-varieties} leads to the desired results.
We begin by settling \cref{i:fin-co-gen-implies-fin-more-sorts}.
\begin{proposition}\label{p:how-finitely-cogenerated}
	For a multiset $X$ the following are equivalent:
	\begin{enumerate}
		\item \label{i:finite}					$X$ is a finite set and every element of $X$ has finite denominator;
		\item \label{i:finitely-co-presentable}	$X$ is finitely co-presentable in $\MS$;
		\item \label{i:finitely-co-generated}	$X$ is finitely co-generated in $\MS$.
	\end{enumerate}
\end{proposition}
\begin{proof}
	
	[\Cref{i:finite} $\Rightarrow$ \cref{i:finitely-co-presentable}]
	This follows from the fact that $\MS$ is the pro-completion of the category of multisets with a finite underlying set and finite denominators (see \cite[Lemma VI.1.8, p.\ 231]{Joh86}).
	
	[\Cref{i:finitely-co-presentable} $\Rightarrow$ \cref{i:finitely-co-generated}]
	This holds trivially.
	
	[\Cref{i:finitely-co-generated} $\Rightarrow$ \cref{i:finite}]
	Let $X$ be a finitely co-generated multiset. 
	By \cite[Theorem 5.1]{CigDubMun04}, every multiset is a limit of a co-filtered system of finite multisets and epic transition arrows. Let $F\colon \mathsf{D}\to \MS$ be such a system for $X$ and let $f_D\colon X\to F(D)$ be the corresponding limit arrows, for $D\in \mathsf{D}$.
	Since $X$ is finitely co-generated, by (the dual version of) \ref{e:factorization} in \cref{d:finitely-presentable}, there exists $D\in \mathsf{D}$ such that the identity $1_X\colon X\to X$ factors through $f_D\colon X\to F_D$ via a suitable arrow $g\colon F_D\to X$.
	Since $F_D$ has finite cardinality and $g$ is surjective, $X$ must also have finite cardinality. Since $F_D$ has finite denominators and $g$ decreases denominators, the multiset $X$ has finite denominators, too.
\end{proof}
We now address \cref{i:reg-co-gen-implies-infin-more-sorts:2} at the beginning of the section. 
\begin{definition}[Dual to \cref{d:generator}]\label{d:co-generator}
A set of objects $\GG$ in a category with products is called a \emph{set of co-generators}  if one of the two following equivalent conditions are satisfied.
\begin{enumerate}[(i)]
\item\label{d:co-generator:item1} For each pair of distinct arrows $f_{1},f_{2}\colon X\rightrightarrows X'$, there exists $G\in \GG$ and an arrow $g\colon X'\to G$ such that $g\circ f_{1}\neq g\circ f_{2}$.
\item\label{d:co-generator:item2} The canonical arrow
\begin{equation}\label{eq:co-generator}
X \to \prod_{G \in \GG,\,t \in \hom(X,G)} G
\end{equation}
is monic.
\end{enumerate}   
\end{definition} 
In \cref{t:how-reg-cogen} we will prove that a set $\GG$ of objects in $\MS$ is regularly co-generating if and only if, for every $n \in \{1\} \cup \{p^k \mid p \in \P, k \in \Nnot\}$, some $G \in \GG$ has two distinct elements of denominator $\nu_n$ and $\nu_1$, respectively.
\begin{notation}
We write $\2$ for the multiset $(\{0,1\},\zeta_{\2})$ where $\{0,1\}$ is endowed with the discrete topology and $\zeta_{\2}(0)=\zeta_{\2}(1)=\nu_1$.
\end{notation}
\begin{lemma}\label{l:2-cogenerator}
The multiset $\2$ is a co-generator in $\MS$.
\end{lemma}
\begin{proof} We check that \cref{d:co-generator:item1} in \cref{d:co-generator} holds.
Let $f_1,f_2\colon X\rightrightarrows X'$ be a pair of distinct arrows in $\MS$.
Then there exists $x \in X$ such that $f_1(x) \neq f_2(x)$.
Since $X'$ is Hausdorff and has a basis of clopen sets, it follows that there exists a clopen subset $C$ of $X'$ such that $f_1(x) \in C$ and $f_2(x) \notin C$.
The characteristic function of $C$ $1_C\colon X\to\2$ is clearly continuous, denominators decreasing, and such that $1_C(f_1(x)) = 1 \neq 0= 1_C(f_2(x))$, so $1_C \circ f_1 \neq 1_C \circ f_2$.
\end{proof} 
\begin{remark}\label{r:projections-jointly-monic}
Recall that the projection arrows from a product are jointly monic, i.e., if $f,g\colon X\to \prod_{i\in I}Y_{i}$ are distinct arrows, then there exists $i\in I$ such that $\pi_{i}\circ f=\pi_{i}\circ g$ (This is a consequence of the fact that each cone on $\{Y_{i}\}_{i\in I}$ factors in a unique way through $\prod_{i\in I}Y_{i}$).
\end{remark}
\begin{lemma}\label{l:char-cogen}
A set of objects $\GG$ in $\MS$ is co-generating if and only if there exists $G\in \GG$ that has at least two distinct points of denominator $\nu_{1}$.
\end{lemma}
\begin{proof}
Suppose $\GG$ is a co-generating set.
Then the canonical arrow $h\colon \2\to \prod_{G\in \GG,\,t \in \hom(\2, G)} G$ is monic so, by \cref{l:arrows-in-MS}, it is injective.  It follows that $h(0)\neq h(1)$, whence there are two distinct maps $f$ and $g$ from the one-element multiset $(\{*\},\zeta_{*})$ where $\zeta_{*}(*)=\nu_{1}$ into the product.
By \cref{r:projections-jointly-monic}, there exist $G\in \GG$ and a projection $\pi$ into $G$ such that $\pi(h(0))\neq\pi(h(1))$.  Therefore, $G$ has two distinct elements; since $\pi\circ h$ decreases denominators, both elements have denominator $\nu_{1}$.

For the converse direction, let us assume that there exists $G\in \GG$ that has at least two distinct points $g_{1}$ and $g_{2}$ with denominator $\nu_{1}$.
We define a map $t \colon \2\to G$ by setting $t(0)\df g_{1}$ and $t(1)\df g_{2}$. The map $t$ is clearly continuous, denominator-decreasing and injective, and hence a monic arrow by \cref{l:arrows-in-MS}. Recall that $\2$ is a co-generator (\cref{l:2-cogenerator}); so, for all distinct arrows $f_{1},f_{2}\colon X\rightrightarrows X'$, there exists an arrow $g\colon X'\to \2$ such that $g\circ f_{1}\neq g\circ f_{2}$.  Since $t$ is monic, the latter inequality holds if and only if  $t\circ g\circ f_{1}\neq t\circ g\circ f_{2}$.  Thus, $G$ is a co-generator, as well. We conclude that $\GG$ is a co-generating set, as it contains a co-generator.
\end{proof}
Next, we will characterise regularly co-generating sets of objects in $\MS$.
Dualising \cref{d:regular-generator}  we say that a set of objects $\GG$ in $\MS$ is \emph{regularly co-generating} if, for every object $X$ in $\MS$, the canonical arrow in \eqref{eq:co-generator} is regular monic.
By \cref{l:arrows-in-MS}, an arrow $f \colon X \to Y$ in $\MS$ is regular monic if and only if it is monic and it preserves denominators.
\Cref{l:char-cogen} gives already a characterisation of the sets $\GG$ for which the map \eqref{eq:co-generator} is monic. So, we now focus on the remaining condition, i.e.\ preservation of denominators. We need some preliminary results.

Given $X$ and $Y$ multisets and arbitrary $x \in X$ and $y \in Y$, a necessary condition for the existence of an arrow $f \colon X \to Y$ such that $f(x) = y$ is that $\zeta_X(x) \geq \zeta_Y(y)$.
The following result establishes a partial converse.
\begin{lemma}\label{l:exists-arrow}
	Let $X$ and $Y$ be multisets and let $x\in X$, $y\in Y$ be such that $\zeta_X(x)\geq\zeta_Y(y)$.
	Suppose that $\zeta_Y(y)$ is finite and that $Y$ has an element of denominator $\nu_1$.
	Then there exists an arrow of multisets $f\colon X\to Y$ such that $f(x)=y$.
\end{lemma}
\begin{proof}
Since $\zeta_Y(y)$ is finite, the set $\{\nu\in \NN\mid \nu\geq \zeta_Y(y)\}$ is open.  Then the set $\left\{z\in X\mid \zeta_X(z)\geq \zeta_Y(y)\right\} = \zeta_{X}^{-1}\left[\{\nu\in \NN\mid \nu\geq \zeta_Y(y)\}\right]$ is open because $\zeta_{X}$ is continuous, so it can be written as a union of clopen sets; let $C$ be one among those which contains $x$. Notice that $\zeta_{X}(z)\geq \zeta_Y(y)$, for any $z\in C$.
Let $y_1$ be an element of $Y$ of denominator $\nu_1$ given by hypothesis.
Let the function $f \colon	X \to	Y$ be defined by:
\begin{equation*}
					f(z)\df		{
											\begin{cases}
												y	& \text{if } z \in C;\\
												y_1	& \text{ otherwise.}
											\end{cases}
											}
\end{equation*}
The function $f$ is continuous because $C$ is clopen; moreover, a case inspection immediately shows that $f$ decreases denominators. The desired conclusion trivially follows.
\end{proof}
The elements of the form $\nu_{p^k}$ play a special role in $\NN$, the following two lemmas capture their main properties. Recall that, in a complete lattice $L$, an element $x$ is said \emph{completely join-irreducible} if, for every subset $J\seq L$, the condition $x=\bigvee J$ implies $x\in J$.
In this definition $J$ is allowed to be empty, so the least element of $L$ is not completely join-irreducible.
\begin{lemma} \label{l:supremum}
	Every $\nu\in \NN$ is the supremum in $\NN$ of the elements under $\nu$ of the form $\nu_{p^k}$, for $p\in \P$ and $k\in \Nnot$.
\end{lemma}
\begin{proof}
Let $\nu$ be an arbitrary element of $\NN$ and let $S_{\nu}\df\{\nu_{p^{k}}\mid \nu_{p^{k}}\leq \nu,\, p\in \P, \, k\in \Nnot \}$.  Obviously $\nu$ is an upper-bound for $S_{\nu}$.  Suppose $\mu$ is another upper-bound for $S_{\nu}$ and $\mu<\nu$.  Then for every $p\in \P$, we have $\mu(p)\leq\nu(p)$ and there exists $p_{0}\in \P$ such that $\mu(p_{0})<\nu(p_{0})$.  Thus, the supernatural number $\bar\mu$ which agrees with $\mu$ on every prime different from $p_{0}$ and attains the value $\mu(p_{0})+1$ on $p_{0}$ is strictly greater than $\mu$ but still in $S_{\nu}$, as $\bar\mu\leq \nu$.  This contradicts the fact that $\mu$ is an upper-bound and the proof is concluded.
\end{proof}
\begin{lemma} \label{l:join-irreducible-elements}
	The completely join-irreducible elements of $\NN$ are precisely the elements of the form $\nu_{p^{k}}$, for $p\in \P$ and $k\in \Nnot$.
\end{lemma}
\begin{proof}
Let $\nu$ be a completely join-irreducible element of $\NN$. By \Cref{l:supremum}, $\nu\in \NN$ is the supremum of the set $\{\nu_{p^{k}}\mid \nu_{p^{k}}\leq \nu,\, p\in \P, \, k\in \Nnot \}$.  Since $\nu$ is completely join-irreducible, $\nu$  belongs to this set and hence $\nu=\nu_{p^{k}}$ for some  $p\in \P$ and $k\in \Nnot$.

For the converse implication, fix $p\in \P$ and $k\in \Nnot$. If $\nu_{p^{k}}=\bigvee J$, then $J$ is linearly ordered and contains only elements of the form $\nu_{p^{j}}$ with $j\leq k$, whence $J$ is finite. Therefore, $\nu_{p^{k}}=\bigvee J$ if and only if $\nu_{p^{k}}\in J$.  As a consequence, every element of the form $\nu_{p^k}$ is completely join-irreducible.
\end{proof}
\begin{notation}
	For any $n\in \Nnot$, let $\Dn$ denote the multiset whose underlying Boolean space is a two-point discrete space $\{0,1\}$, with $\zeta(0)\df\nu_{1}$ and $\zeta(1)\df\nu_n$. Notice that, with our previous notation, $\2$ coincides with $\D{1}$. 
\end{notation}
\begin{lemma}\label{l:char-reg-cogen}
Let $\GG$ be a set of multisets. The following conditions are equivalent.
\begin{enumerate}[(i)]
\item \label{i:preserves} For every multiset $X$, the canonical arrow from $X$ to $\prod_{G\in \GG,\,t \in \hom(X,G)} G$ preserves denominators. 
\item \label{i:exists} For every $p \in \P$ and $k \in \Nnot$, there exists $G\in \GG$ that has at least one point of denominator $\nu_{p^k}$ and one point of denominator $\nu_1$.
\end{enumerate}
\end{lemma}
\begin{proof}
Throughout the proof, let us denote by $\zeta_{\prod}$ the denominator map of $\prod_{G\in \GG,\,t \in \hom(\D{p^k}, G)} G$.

Suppose \cref{i:preserves} holds.
Let $p \in \P$ and $k \in \Nnot$.
Consider the canonical arrow
\[h\colon \D{p^k}\to \prod_{G\in \GG,\,t \in \hom(\D{p^k}, G)} G.\]
We have
\[\nu_{p^k} = \zeta_{\D{p^k}}(1) \stackrel{\text{\cref{i:preserves}}}{=} \zeta_{\prod}(h(1))\stackrel{\text{\cref{l:limits}}}{=} \bigvee_{G\in \GG,\,t \in \hom(\D{p^k},G)} \zeta_G(t(1)).\]
By \cref{l:join-irreducible-elements}, the element $\nu_{p^k}$ is completely join-irreducible; therefore, there exist $G \in \GG$ and $t\in \hom(\D{p^k},G)$ such that $\zeta_G(t(1))=\nu_{p^k}$. Furthermore, $\zeta_G(t(0))=\nu_1$.
Since $p$ and $k$ are arbitrary, we conclude that \ref{i:preserves} $\Rightarrow$ \ref{i:exists}.

Let us prove the converse implication: suppose \cref{i:exists} holds. Let $X$ be a multiset and $h$ be the canonical arrow from $X$ to $\prod_{G\in \GG,\,t \in \hom(X,G)}G$.
Then
\begin{align*}
\zeta_{\prod}(h(x))	& = \bigvee_{G\in \GG,\, t\in \hom(X,G)} \zeta_G(t(x)) \\
& = \bigvee_{G\in \GG,\, t\in \hom(X,G)} \bigvee_{\nu_{p^k} \leq \zeta_G(t(x))}\nu_{p^k} && (\text{\cref{l:supremum}}) \\
& = \bigvee \{\nu_{p^k} \mid \exists G \in \GG \ \exists t \in \hom(X,G) \text{ s.t.\ } \nu_{p^k} \leq \zeta_G(t(x))\}  \\
& = \bigvee \{\nu_{p^k} \mid \nu_{p^k} \leq \zeta_G(x)\} && (\text{\cref{l:exists-arrow}}) \\
& = \zeta_G(x).
\end{align*}
Thus, the map $h$ preserves denominators.
\end{proof}
Combining \cref{l:char-cogen,l:char-reg-cogen} we obtain a characterisation of regularly co-generating sets of objects in $\MS$.
\begin{theorem}\label{t:how-reg-cogen}
A set of objects $\GG$ in $\MS$ is regularly co-generating if and only if, for every $n \in \{1\} \cup \{p^k \mid p \in \P, k \in \Nnot\}$, there exists $G \in \GG$ with two distinct elements, one having denominator $\nu_{n}$ and the other one $\nu_1$.
\end{theorem}
\begin{proof}
A set of objects $\GG$ in $\MS$ is regularly co-generating if and only if the canonical arrow $h$ from $X$ to $\prod_{G\in \GG,\,t \in \hom(X,G)} G$ is regular monic.
By \cref{l:arrows-in-MS}, an arrow in $\MS$ is regular monic if and only if it is injective and denominator-preserving.
By \cref{l:char-cogen}, the map $h$ is injective if and only if there exists $G \in \GG$ that has at least two distinct points of denominator $\nu_1$.
By \cref{l:char-reg-cogen}, the map $h$ preserves denominators if and only if, for every $p \in \P$ and $k \in \Nnot$, there exists $G \in \GG$ that has at least one point of denominator $\nu_{p^k}$ and one point of denominator $\nu_1$.
\end{proof}
\begin{corollary}\label{c:reg-cogen-inf}
Every regularly co-generating set in $\MS$ has either infinitely many objects or an object whose underlying set is infinite.
\end{corollary}
If we replace the notion of a co-generating set with the notion of a co-generator, we obtain the following result.
\begin{corollary}\label{c:reg-cogen}
An object $G$ is a regular co-generator in $\MS$ if and only if, for every $n \in \{1\} \cup \{p^k \mid p \in \P, k \in \Nnot\}$, there exists $g\in G$ that has denominator $\nu_{n}$. In particular, the underlying set of a regular co-generator in $\MS$ is infinite.
\end{corollary}

As a consequence of \cref{t:how-reg-cogen,c:reg-cogen-inf}, we have the following.
\begin{proposition} \label{p:no-finite-set}
	The category $\MVlf$ of locally finite MV-algebras has no regularly generating finite set of finitely generated objects.
\end{proposition}
\begin{proof}
	We prove the dual statement.
	By \cref{p:how-finitely-cogenerated}, every finitely co-generated object in $\MS$ has finite underlying set.
	By \cref{c:reg-cogen-inf}, no finite set of multisets with finite underlying set is a regularly co-generating set in $\MS$.
	Thus, $\MS$ has no regularly co-generating finite set of finitely co-generated objects.
\end{proof}
\begin{theorem}\label{t:locally-finite-are-not-finitely-sorted-quasi-variety}
	The category of locally finite MV-algebras is not equivalent to any finitely-sorted quasi-variety of finitary algebras.
\end{theorem}
\begin{proof}
	From \cref{p:no-finite-set} we know that the category $\MVlf$ has no regularly generating finite set of finitely generated objects.
	Now we use the characterisation of many-sorted quasi-varieties of finitary algebras (\cref{t:quasi-variety-many-sorted} in \cref{t:char-quasi-varieties}), together with \cref{l:abstractly-finite-vs-finitely-generated}.
	By  \cref{r:card-sorts-arity}, the number of objects in the regularly generating set corresponds to the number of sorts.
	It follows that $\MVlf$ is not equivalent to a finitely-sorted quasi-variety of finitary algebras.
\end{proof}
\begin{corollary}\label{c:locally-finite-are-not-finitary-variety}
	The category of locally finite MV-algebras is not equivalent to any variety of finitary algebras.
\end{corollary}
%

\section{The case of an infinite language}
\label{s:infinite-language}


In this section, we will prove that $\MVlf$ is equivalent to a quasi-variety of algebras if one allows either a language with operations of at most countable arity or a language with countably many sorts.  The key point is to prove that the objects $\Dn$ defined in the previous section are all regular injective.  

Dualising the notion of regular projective, we obtain the notion of regular injective.
Explicitly, an object $X$ is called \emph{regular injective} if, for every regular monic arrow $g\colon B \rmono A$ and every arrow $f\colon B\to X$, there exists an arrow $h\colon A \to X$ such that the following diagram commutes:
\[
	\begin{tikzpicture}
	
		\node (B) at (0,0)	{$B$};
		\node (A) at (0,1.5)	{$A$};
		\node (P) at (1.5,1.5)		{$X$};

		\draw [right hook->]	(B) -- (A) node [left, midway] {$g$};
		\draw [->]				(B) -- (P) node [right, midway] {$f$};
		\draw [->, dashed]	(A) -- (P) node [above, midway] {$h$};
	\end{tikzpicture}
\]

It is known that in a Boolean space $X$ if $A\seq B\seq X$ with $A$ closed and $B$ open, then there exists a clopen subset $C$ of $X$ such that $A \seq C \seq B$. The following lemma extends this result to a family of pairs of open and closed sets, in such a way that the $C_{i}$ form a partition of $X$.
\begin{lemma}\label{l:partition}
	Let $X$ be a Boolean space, let $K_1, \dots, K_n$ be pairwise disjoint closed subsets of $X$ and let $Z_1, \dots, Z_n$ be an open cover of $X$ such that $K_i \seq Z_i$ for $i \in \{1, \dots, n\}$.
	Then $X$ can be partitioned into clopen subsets $C_1, \dots, C_n$ such that $K_i \seq C_i \seq Z_i$, for $i \in \{1, \dots, n\}$.
\end{lemma}
\begin{proof}
Without loss of generality we can assume that for all $i\neq j$ the sets $K_i$ and $Z_j$ are disjoint. Indeed, if this is not the case, the set $Z_j$ can be replaced with the smaller $Z_{j} \setminus \bigcup_{i \in \{1, \dots, n\} \setminus\{j\}} K_i$ for each $j \in \{1, \dots, n\}$, and the latter satisfies this additional hypothesis.
We now prove the statement of the lemma by induction on $n$. The case $n = 1$ is trivial.
Let $n \geq 2$ and suppose that the property holds for $n-1$.  
Let $L_{n}\coloneqq X \setminus (Z_1 \cup \dots \cup Z_{n-1})$;
since $Z_1$, \ldots, $Z_n$ cover $X$, it holds that $L_{n} \seq Z_n$. Moreover, by hypothesis $K_n \seq Z_n$. 
Therefore, $K_n \cup L_{n} \seq Z_n$, the set $K_n \cup L_{n}$ is closed, and  $Z_n$ is open.
Hence, there exists a clopen subset $C_n$ of $X$ such that $K_n \cup L_{n}\seq C_n \seq Z_n$.
Since $C_n \seq Z_n$, it follows from our additional hypothesis that the set $C_n$ is disjoint from $K_1$, \ldots, $K_{n-1}$.
So, $Y \df X \setminus C_n$ is a Boolean space, the sets $K_1$, \ldots, $K_{n-1}$ are closed and disjoint subsets of $Y$, the sets $Z_1 \cap Y$, \ldots, $Z_{n-1} \cap Y$ are an open cover of $Y$ and $K_i \seq Z_i\cap Y$, for every $i \in \{1, \dots, n-1\}$.
By inductive hypothesis, the space $Y$ can be partitioned into clopen subsets $C_1, \dots, C_{n-1}$ such that $K_i \seq C_i \seq Z_{i} \cap Y$, for every $i \in \{1, \dots, n-1\}$.
The sets $C_1, \dots, C_n$ are clopen subsets of $X$ and they are a partition of $X$.
Furthermore,  $K_i \seq C_i \seq Z_i$, for every $i \in \{1, \dots, n\}$.
This concludes the proof.
\end{proof} 
\begin{lemma}\label{l:Dn-is-regular-injective}
	Let $X$ be a multiset and suppose that the following conditions hold.
	\begin{enumerate}
		\item\label{l:Dn-is-regular-injective:item1} The set $X$ is finite and every element of $X$ has a finite denominator.
		\item\label{l:Dn-is-regular-injective:item3} There exists an element $x_0\in X$ with denominator $\nu_1$.
	\end{enumerate}
	Then $X$ is regular injective in $\MS$.
\end{lemma}
\begin{proof}
Let $g\colon B\rmono A$ be any regular monic arrow in $\MS$ and let $f\colon B\to X$ be an arrow in $\MS$.
	We will prove that there exists an arrow $h$ in $\MS$ such that the following diagram commutes:
	\[
		\begin{tikzpicture}
		
			\node (B) at (0,0)	{$B$};
			\node (A) at (0,1.5)	{$A$};
			\node (P) at (1.5,1.5)		{$X$};
	
			\draw [right hook->]	(B) -- (A) node [left, midway] {$g$};
			\draw [->]				(B) -- (P) node [right, midway] {$f$};
			\draw [->, dashed]	(A) -- (P) node [above, midway] {$h$};
		\end{tikzpicture}
	\]
	Let $x_1,\dots,x_n$ be a listing of the elements of $X\setminus \{x_0\}$.
	Let $d_1,\dots,d_n$ be the denominators of $x_1,\dots,x_n$, respectively.
	Fix $i\in \{0,\dots,n\}$. Since $X$ is a finite Boolean space, the singleton $\{x_i\}$ is a clopen subset of $X$.
	Thus, $f^{-1}[\{x_i\}]$ is a clopen subset of $B$.
	The function $g$ is closed by the Closed Map Lemma; therefore, the set $K_i \coloneqq g[f^{-1}[\{x_i\}]]$ is closed in $A$.
	Since $g$ is monic, the sets $K_0,\dots,K_n$ are disjoint.
	For every $i\in \{0,\dots,n\}$, set 
	\[Z_i\coloneqq\{z\in A\mid \zeta_A(z)\geq d_i \}=\zeta_A^{-1}[\{\nu\in \NN\mid \nu\geq d_i\}].\] 
	The set $\{\nu\in \NN\mid \nu\geq d_i\}$ is an open subset of $\NN$ because $d_{i}$ is finite by \cref{l:Dn-is-regular-injective:item1} so, for every $i \in \{0, \dots, n\}$, the set $Z_i$ is an open subset of $A$ because $\zeta_A$ is continuous.
	Notice that $\zeta_{B}(a)\geq d_{i}$ for all $a\in f^{-1}[\{x_{i}\}]$ and, since $g$ is regular monic, all denominators of $K_{i}$ are greater or equal to $d_{i}$.  It follows that $K_{i}\seq Z_{i}$.
	Moreover, $Z_0=A$ by \cref{l:Dn-is-regular-injective:item3}.
	By \cref{l:partition}, the set $A$ can be partitioned into clopen subsets $C_0,\dots,C_n$ such that $K_i\seq C_i\seq Z_i$, for every $i\in \{0,\dots,n\}$.
	Define $h\colon A\to X$ by setting $h(c)=x_i$ for $c\in C_i$.
	The function $h$ is trivially continuous.
	To see that it decreases denominators, let $i\in \{0,\dots,n\}$ and $c\in C_i$.
	Then $\zeta(h(c))=\zeta(x_i)=d_i$.
	Since $c\in C_i\seq Z_i$, it follows that $\zeta_A(c)\geq d_i$.
	Thus, the function $h$ decreases denominators.
	Finally, from $g\left[ f^{-1}[\{x_i\}]\right]\seq C_i$, it is easy to derive $f=h\circ g$.
\end{proof}
\begin{remark}\label{r:countably}
	Let $\{X_i\}_{i\in I}$ be a family of compact topological spaces.
	Each clopen subset of the topological product $\prod_{i\in I}X_i$ is nontrivial only on finitely many coordinates $i\in I$.
	Thus, if $Y$ is a finite discrete space, every continuous function from $\prod_{i\in I}X_i$ into $Y$ depends on finitely many coordinates.
	Furthermore, it follows that if $\{Y_n\}_{n\in \N}$ is a countable family of finite discrete spaces, any continuous function from $\prod_{i\in I}X_i$ to $\prod_{n\in \N}Y_n$ depends on at most countably many coordinates $i\in I$.
\end{remark}
\begin{proposition}\label{p:sorted-quasi-variety}
	The set $\{\Dn \mid n \in \{p^k \mid p \in \P, k \in \Nnot\}\cup \{1\}\}$ is an abstractly co-finite regularly co-generating countable set of regular injective objects.  Hence the category of locally finite MV-algebras is equivalent to a countably-sorted quasi-variety of finitary algebras.
\end{proposition}
\begin{proof}
The set $\{\Dn \mid n \in \{p^k \mid p \in \P, k \in \Nnot\}\cup\{1\}\}$ is abstractly co-finite by \cref{r:countably}, it is regularly co-generating by \cref{t:how-reg-cogen}, and each $\Dn$ is regular injective by \cref{l:Dn-is-regular-injective}.
In addition, $\MVlf$ is co-complete by \cref{c:complete-and-co}; so, by \cref{t:quasi-variety-many-sorted} in \cref{t:char-quasi-varieties}, $\MVlf$ is equivalent to a countably-sorted quasi-variety of finitary algebras.
\end{proof}
The countably-sorted quasi-variety in \cref{p:sorted-quasi-variety} is actually a countably-sorted variety, as we will prove in \cref{t:sorted-variety}.
\begin{lemma}\label{l:reg-inj-reg-cog}
	The multiset 
	\[C\df\prod_{n\in \{ p^k \mid p \in \P, k \in \Nnot\}\cup \{1\}}\Dn\] 
	is a regular injective regular co-generator of $\MS$.
	Moreover, every arrow from a power of $C$ to $C$ factors through an at most countable sub-power.
\end{lemma}
\begin{proof}
	Regular injective objects are closed under products (see, e.g.\ \cite[Lemma 5.13]{AlgThe}), so \cref{l:Dn-is-regular-injective} entails that $C$ is regular injective.
	Moreover, $C$ is a regular co-generator by \cref{c:reg-cogen}.
	Finally, by \cref{r:countably}, every arrow from a power of $C$ to $C$ factors through an at most countable sub-power.
\end{proof}
\begin{proposition}\label{p:MVlf-quasi-variety}
	The category of locally finite MV-algebras is equivalent to a quasi-variety of algebras, with operations of at most countable arity. 
\end{proposition}
\begin{proof}
	By \cref{c:complete-and-co}, $\MVlf$ is co-complete.
	By \cref{l:reg-inj-reg-cog}, $\MVlf$ has a regular projective regular generator $G$ such that every arrow from $G$ to a co-power of $G$ factors through an at most countable sub-co-power.
	Therefore, by \cref{t:char-quasi-varieties}, $\MVlf$ is equivalent to a quasi-variety of algebras and, by \cref{r:card-sorts-arity}, the operations depend on at most countably many coordinates.
\end{proof}
The quasi-variety in \cref{p:MVlf-quasi-variety} is actually a variety, as we will prove in \cref{t:MVlf-variety}.
%

\section{From quasi-varieties to varieties: effective equivalence relations}
\label{s:from-quasi-variety-to-variety}

In this section, we prove that all reflexive (internal) relations in $\MVlf$ are effective equivalence relations.
To study equivalence relations in $\MVlf$ we look at their dual in $\MS$.
To do so, we need to introduce some notation and tools for the dual concepts. 

By a \emph{co-subobject} we mean the dual notion of subobject: a co-subobject on $X$ is an equivalence class of epic arrows with domain $X$.
With a little abuse of notation, we refer to a co-subobject just by one of its representatives.
By a \emph{(binary) co-relation on $X$} we mean a co-subobject on $X+X$.  Notice that arrows with domain $X+X$ bijectively correspond to pairs of parallel arrows with domain $X$.
Dualising the respective notion for relations, we say that a co-relation $q_0, q_1 \colon X \rightrightarrows S$ on $X$ is \emph{reflexive} if  there exists an arrow $d \colon S \to X$ such that the following two diagrams commute:
\begin{figure}[H]
\centering
\begin{tikzcd}
	X \arrow{r}{q_0} \arrow[swap]{rd}{1_X}	& S \arrow[dashed]{d}{d} \\
											& X
\end{tikzcd}
\hspace{2cm}
\begin{tikzcd}
	X \arrow{r}{q_1} \arrow[swap]{rd}{1_X}	& S \arrow[dashed]{d}{d} \\
											& X
\end{tikzcd}
\end{figure}
\emph{Symmetric}, \emph{transitive} and \emph{equivalence} co-relations are obtained in a similar way.
Dualising \Cref{d:effective-exact}, we say that an equivalence co-relation $q_0, q_1 \colon X \rightrightarrows S$ on the multiset $X$ is \emph{effective} if it is the co-kernel pair of some arrow.
\begin{lemma}\label{l:char-co-refl}
If a co-relation $q_0, q_1 \colon X \rightrightarrows S$ on $X$ in $\MS$ is reflexive, then for all $i, j \in \{0,1\}$ and $x,y\in X$:
\begin{enumerate}
\item $q_i(x)=q_j(y)\text{ implies }  x=y$,
\item $\zeta_{S}(q_i(x))=\zeta_{X}(x).$
\end{enumerate}
\end{lemma}
\begin{proof}
By reflexivity, there exists a map $d \colon S \to X$ such that $d \circ q_0 = 1_X$, and $d \circ q_ 1 = 1_X$.
Therefore, if $q_i(x)=q_j(y)$, then $x = d(q_i(x)) = d(q_j(x)) = y$.

Since $q_i$ is an arrow of multisets, we have $\zeta_S(q_i(x)) \leq \zeta_X(x)$.
For the converse inequality, since $d$ decreases denominators, we have $\zeta_S(q_i(x)) \geq \zeta_X(d(q_i(x))) = \zeta_X(x)$.
Therefore, $\zeta_S(q_i(x)) = \zeta_X(x)$.
\end{proof}

It should be noted that also the converse of the previous lemma holds, thus providing a characterisation of reflexive co-relations in $\MS$. For a proof see \cref{l:char-refl} in the Appendix.
\begin{proposition} \label{p:refl-implies-effective}
Every reflexive co-relation in $\MS$ is an effective equivalence co-relation.
\end{proposition}
\begin{proof}
	Let $q_0, q_1 \colon X \rightrightarrows S$ be a reflexive co-relation on a multiset $X$.
	Set $Y\df \{x\in X\mid q_0(x)=q_1(x) \}$.
	By \cref{l:char-co-refl}, we have $q_i(x)=q_j(y) \Rightarrow  x = y$.
	Therefore, for every $x, y \in X$ and every $i \in \{0,1\}$, 
	\begin{enumerate}
		\item $q_0(x) = q_1 (y)$ if and only if $x \in Y$, $y \in Y$ and $x = y$.
		\item $q_i(x) = q_i (y)$ if and only if $x = y$.
	\end{enumerate}
	Therefore, the following is a pushout in $\Set$:
	\begin{equation}\label{e:pushout-effective}
	\begin{tikzcd}
	Y \arrow[hookrightarrow]{r}{}& X \arrow{d}{q_0} \\
	X \arrow[hookleftarrow]{u}{} \arrow[swap]{r}{q_1}& S
	\end{tikzcd}
	\end{equation}
	Equip $Y$ with the subspace topology.
	By \cref{l:reflect-colimits}, the forgetful functor from $\BS$ to $\Set$ reflects co-limits, and therefore \eqref{e:pushout-effective} is a pushout in $\BS$, as well.
	Equip $Y$ with the the denominator map $\zeta_Y \colon Y \to \NN$, $y \mapsto \zeta_X(y)$.
	By \cref{l:char-co-refl}, $\zeta_{S}(q_i(x))=\zeta_{X}(x)$.
	Hence, by the construction of finite co-limits (see \cref{l:colimits-MS}) in $\MS$, \eqref{e:pushout-effective} is a pushout also in $\MS$. 
\end{proof}
\begin{corollary}\label{c:refl-is-effective}
Every reflexive relation in $\MVlf$ is an effective equivalence relation. Thus in particular $\MVlf$ is a Mal'tsev category\footnote{Recall that a category is called \emph{Mal'tsev} if it has finite limits and every reflexive relation is an equivalence relation.}.
\end{corollary}
\begin{theorem}\label{t:MVlf-variety}
The category of locally finite MV-algebras is equivalent to a variety of algebras, with operations of at most countable arity.
\end{theorem}
\begin{proof}
By \cref{p:MVlf-quasi-variety}, $\MVlf$ is equivalent to a quasi-variety of algebras, with operations of at most countable arity.
By \cref{c:refl-is-effective}, every reflexive relation in $\MVlf$ is an effective equivalence relation.
The statement of the theorem follows by an application of \cref{t:char-quasi-varieties}.
\end{proof}
\begin{theorem}\label{t:sorted-variety}
The category of locally finite MV-algebras is equivalent to a countably-sorted variety of finitary algebras.
\end{theorem}
\begin{proof}
By \cref{p:MVlf-quasi-variety}, $\MVlf$ is equivalent to a countably-sorted quasi-variety of finitary algebras.
By \cref{c:refl-is-effective}, every reflexive relation in $\MVlf$ is an effective equivalence relation.
The statement of the theorem follows by an application of \cref{t:char-quasi-varieties}.
\end{proof}
We conclude the main development of our work by pointing out that from our results it follows that $\MVlf$ is an exact category (in the sense of Barr), i.e., a regular category in which every equivalence relation is effective.  For more details on regular and exact categories, we refer the reader to \cite{BarrGrilletvanOsdol1971}.


\section{Appendix}
\label{s:appendix}


\subsection{The topology of \texorpdfstring{$\NN$}{N}}
\label{ss:topology-on-N}

In \cite{CigDubMun04}, the authors erroneously state that the topology on $\NN$ having as an open basis the sets of the form $\{\nu\in \NN\mid\nu> \nu_n\}$  for $n\in \Nnot$ coincides with the topology having as an open subbasis the sets of the form $\{\nu\in \NN\mid\nu(p)> k\}$ for $p\in \P$ and $k\in \N$.  However, the error does not appear in \cite[Section 8.4]{Mun2011}.
The statement in \cite{CigDubMun04} is incorrect for two reasons.
First, the sets of the form $\{\nu\in \NN\mid\nu> \nu_n\}$  for $n\in \Nnot$ do not form a basis for any topology on $\NN$.
Indeed, these sets do not cover $\NN$ because, for every $n\in \Nnot$, we have $\nu_1\notin\{\nu\in \NN\mid\nu> \nu_n\}$.
Secondly: even if we replace ``basis'' by ``subbasis'', the statement remains incorrect: the topology on $\NN$ having as an open subbasis the sets of the form $\{\nu\in \NN\mid\nu> \nu_n\}$  for $n\in \Nnot$ does not coincide with the topology having as an open subbasis the sets of the form $\{\nu\in \NN\mid\nu(p)> k\}$ for $p\in \P$ and $k \in \N$.
This is implied by the statement $\tau_1\not\subseteq \tau_4$ in the following lemma.
\begin{lemma}\label{l:various-top}
Let $\tau_1$, $\tau_2$, $\tau_3$, $\tau_4$, $\tau_5$, $\tau_6$ be the topologies on $\NN$ generated by the following families of sets, respectively:
\begin{enumerate}
\item $\{\nu\in \NN\mid\nu(p)> k\}$ for $p\in \P$, $k\in \N$;
\item $\{\nu\in \NN\mid\nu(p)\geq k\}$ for $p\in \P$, $k\in \N$;
\item $\{\nu\in \NN\mid\nu\geq \nu_n\}$ for $n\in \Nnot$;
\item $\{\nu\in \NN\mid\nu> \nu_n\}$  for $n\in \Nnot$;
\item $\{\nu\in \NN\mid\nu\not\leq \nu_n\}$ for $n\in \Nnot$;
\item $\{\nu\in \NN\mid\nu\not< \nu_n\}$ for $n\in \Nnot$.
\end{enumerate}
Then $\tau_1 = \tau_2 = \tau_3$, and they all coincide with the topology we equipped $\NN$ with in \Cref{s:preliminaries}. Furthermore,  $\tau_5=\tau_6$, $\tau_4 \subsetneq \tau_1$ and $\tau_5\subsetneq\tau_1$, and the topologies $\tau_4$ and $\tau_5$ are not comparable.
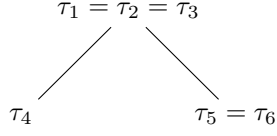
\begin{center}
\begin{tikzpicture}[node distance=2cm]
\node(1)								{$\tau_1 = \tau_2 = \tau_3$};
\node(4)	[below left of=1]	{$\tau_4$};
\node(5)	[below right of=1]	{$\tau_5 = \tau_6$};

\draw(1) -- (4);
\draw(1) -- (5);

\end{tikzpicture}
\end{center}
\end{lemma}
\begin{proof}
\ 
\begin{description}
\item[$\tau_2\subseteq \tau_1$]
If $k\neq 0$, then $\{\nu\in \NN\mid\nu(p)\geq k\}=\{\nu\in \NN\mid\nu(p)> k-1\}$.
If $k=0$, then $\{\nu\in \NN\mid\nu(p)\geq k\}=\NN$.
\item[$\tau_3\subseteq \tau_2$]
Let $n=p_1^{k_1}\cdot\dots \cdot p_l^{k_l}$ be the unique prime factorisation of $n$.
Then 
\[
\{\nu\in \NN\mid\nu\geq \nu_n\}=\bigcap_{i=1}^l \{\nu\in \NN\mid\nu(p_i)\geq k_i\}.
\]
\item[$\tau_1\subseteq \tau_3$]
We have $\{\nu\in \NN\mid\nu(p)> k\}=\{\nu\in \NN\mid\nu\geq \nu_{p^{k+1}}\}$.
\item[$\tau_4\subseteq \tau_3$]
We have $\{\nu\in \NN\mid\nu> \nu_n\}=\bigcup_{m \in \Nnot: \nu_m>\nu_n} \{\nu\in \NN\mid\nu\geq \nu_m\}$.
\item[$\tau_5\subseteq \tau_1$]
We have $\{\nu\in \NN\mid\nu\not\leq \nu_n\}=\bigcup_{p\in \P}\{\nu\in \NN\mid \nu(p)>\nu_n(p) \}$.  
\item[$\tau_6\subseteq \tau_5$]
For all $\nu\in \NN$ and $n \in \Nnot$, the condition $\nu<\nu_n$ holds if and only if there exists $m \in \Nnot$ with $\nu_m<\nu_n$ such that $\nu \leq \nu_m$.
Contrapositively, the condition $\nu \not< \nu_n$ holds if and only if, for every $m \in \Nnot$ with $\nu_m < \nu_n$, we have $\nu \not\leq \nu_m$.
Then 
\[
\{\nu\in \NN\mid\nu\not< \nu_n\}=\bigcap_{m\in \Nnot:\nu_m<\nu_n}\{\nu\in \NN\mid\nu\not\leq \nu_m\}.
\]
Note that the intersection is finite.

\item[$\tau_5\subseteq \tau_6$]
Since $\{\nu\in \NN\mid\nu \leq \nu_n\}=\bigcap_{p\in \P}\{\nu\in \NN\mid\nu< \nu_{np}\}$, it follows
\[
\{\nu\in \NN\mid\nu\not\leq \nu_n\}=\bigcup_{p\in \P}\{\nu\in \NN\mid\nu\not< \nu_{np}\}.
\]

\item[$\tau_4\not\subseteq \tau_5$]
For every $n \in \Nnot$, the set $\{\nu \in \NN \mid \nu \not\leq \nu_n\} = \NN \setminus \{\nu \in \NN \mid \nu \leq \nu_n\}$ is co-finite.
Therefore, $\tau_5$ is generated by co-finite sets; thus, every nonempty element in $\tau_5$ is co-finite.
The topology $\tau_4$ contains nonempty elements which are not co-finite, such as $\{\nu\in \N\mid\nu>\nu_{2}\}$: indeed, for every $m \in \P \setminus \{2\}$, the element $\nu_m$ belongs to $\NN \setminus \{\nu\in \N\mid\nu>\nu_{2}\}$.

\item[$\tau_5\not\subseteq \tau_4$]
We have $\NN\setminus\{\nu_1,\nu_2\}=\{\nu\in \NN\mid \nu\not\leq \nu_2\}\in \tau_5$.
We claim that $\NN\setminus\{\nu_1,\nu_2\}\notin \tau_4$.
We show that there is no finite list $n_1,\dots,n_k$ of elements of $\Nnot$ such that $\nu_3\in \bigcap_{i=1}^k\{\nu\in \NN\mid\nu> \nu_{n_i}\}\subseteq \NN\setminus\{\nu_1,\nu_2\}$.
Indeed, suppose by way of contradiction that such a list exists.
Then for every $i\in \{1,\dots,k\}$,  $\nu_3\in \{\nu\in \NN\mid\nu> \nu_{n_i}\}$, i.e.\ $\nu_3>\nu_{n_i}$, that in turn implies.\ $n_i=1$.
Thus, 
\begin{equation*}
\bigcap_{i=1}^k\{\nu\in \NN\mid\nu> \nu_{n_i}\}=\bigcap_{i=1}^k\{\nu\in \NN\mid\nu> \nu_{1}\}=\bigcap_{i=1}^k\NN\setminus\{\nu_1\} = \begin{cases}
\NN \setminus \{\nu_1\} & \mbox{if } k \neq 0;\\
\NN							& \mbox{if } k = 0.
\end{cases}
\end{equation*}
This set contains the element $\nu_2$, whence it is not included in $\NN \setminus \{\nu_1, \nu_2\}$: a contradiction.

\item[$\tau_1\not\subseteq \tau_4$]
This statement follows from the previous one; here is a direct proof.
Set $A\coloneqq \{\nu\in\NN\mid \nu(2)>0 \}$ and notice that $A\in \tau_1$.
We claim that $A\notin \tau_4$.
The element $\nu_2$ belongs to $A$, in contrast we show that there is no finite list $n_1,\dots,n_k$ of elements of $\Nnot$ such that $\nu_2\in \bigcap_{i=1}^k\{\nu\in \NN\mid\nu> \nu_{n_i}\}\subseteq A$.
Indeed, suppose by way of contradiction that such a list exists.
Then for every $i\in \{1,\dots,k\}$, it holds that $\nu_2\in \{\nu\in \NN\mid\nu> \nu_{n_i}\}$, i.e.\ $\nu_2>\nu_{n_i}$, which in turn implies $n_i=1$.
Thus,
\begin{equation*}
\bigcap_{i=1}^k\{\nu\in \NN\mid\nu> \nu_{n_i}\}=\bigcap_{i=1}^k\{\nu\in \NN\mid\nu> \nu_{1}\}=\bigcap_{i=1}^k\NN\setminus\{\nu_1\} = \begin{cases}
\NN \setminus \{\nu_1\} & \mbox{if } k \neq 0;\\
\NN							& \mbox{if } k = 0.
\end{cases}
\end{equation*}
This set contains $\nu_3$, thus it is not included in $A$ (because $\nu_3 \notin A$): a contradiction.\qedhere
\end{description}
\end{proof}

\begin{remark}
By \cref{l:various-top}, for every $p\in \P$, every $k\in \N$ and every $n\in \Nnot$, the following subsets of $\NN$ are open in the topology $\tau_{1}$:
\begin{enumerate}
\item $\{\nu\in \NN\mid\nu(p)> k\}$;
\item $\{\nu\in \NN\mid\nu(p)\geq k\}$;
\item $\{\nu\in \NN\mid\nu\geq \nu_n\}$;
\item $\{\nu\in \NN\mid\nu> \nu_n\}$;
\item $\{\nu\in \NN\mid\nu\not\leq \nu_n\}$;
\item $\{\nu\in \NN\mid\nu\not< \nu_n\}$.
\end{enumerate}
\end{remark}

Recall that an ideal of a lattice is a subset that is downward closed and closed under finite joins. Notice that $\NN$ can be identified with the ``ideal-completion'' of $\Nnot$ with the divisibility order.  By \cite[Lemma 1.1 and 1.2]{Hoffmann:1979aa}, $\NN$ can be regarded as the ``sobrification'' of $\Nnot$ endowed with the Alexandrov topology induced by the divisibility order\footnote{We are indebted to A.\ Moshier for this remark.}.

The topology $\tau_{1}$ on $\NN$ is not Hausdorff. So $\NN$ with the identity function is not a multiset, which seems somewhat unnatural. However, $\NN$ can be made into a Boolean space by taking the ``patch topology'', i.e., taking as an open subbasis the family of subsets 
\[U_{p,k}\df\{\nu\in \NN\mid\nu(p)> k\} \text{ and } \NN\setminus U_{p,k}\df\{\nu\in \NN\mid\nu(p)\leq k\}\qquad\text{ for }p\in \P, k\in \N.\]
We let $\bar{\NN}$ denote the set $\NN$ with the patch topology.

The multiset $\bar{\NN}$ is not a regular co-generator, because it does not contain two distinct points of denominator $\nu_1$. However, $\bar{\NN}$ is ``almost a regular co-generator'' because, for every $p\in \P$ and $k \in \N$, it contains an element of denominator $\nu_{p^k}$, thus satisfying   \cref{i:exists} of \cref{l:char-reg-cogen}.
Moreover, one can show $\bar{\NN}$ to be regular injective. The set of multisets $\{\2, \bar{\NN}\}$ is then a regularly co-generating set of regular injective objects, and the multiset $\bar{\NN} \times \2$ is a regular injective regular co-generator object of $\MS$.


\subsection{Representation of co-subobjects}\label{ss:adjunction}

Recall that, given a multiset $X$, a \emph{co-subobject of $X$} is an equivalence class of epic arrows of multisets with domain $X$, where two epic arrows $f\colon X\epi Y$ and $g\colon X\epi Z$ are equivalent if there exists an iso $h \colon Y \to Z$ such that $g = h \circ f$. 
We denote with $\Quot(X)$ the set of co-subobjects of $X$ and we equip $\Quot(X)$ with a partial order defined as follows: the equivalence class of $f\colon X\epi Y$ is below the equivalence class of $g\colon X\epi Z$ if and only if there exists a (necessarily unique and epic) arrow $h \colon Y \to Z$ such that the following diagram commutes:
\[
\begin{tikzcd}
X \arrow[two heads]{r}{f} \arrow[swap, two heads]{rd}{g}& Y \arrow[dashed]{d}{h} \\
& Z
\end{tikzcd}
\]

Our next goal is to represent co-subobjects on $X$ internally on $X$.
We recall that this is possible for Boolean spaces thanks to the notion of \emph{Boolean relation}. 
\begin{definition} \label{d:Boolean-rel}
Let $R$ be an equivalence relation on a Boolean space $X$.
A subset $P$ of $X$ is said to be \emph{compatible} with $R$ if $P$ is a union of equivalence classes of $R$. In other words, if an element $x$ is in $P$, then the entire equivalence class of $x$ (modulo $R$) is included in $P$.
The equivalence relation $R$ is called \emph{Boolean} if for any two distinct equivalence classes of $R$, there is a clopen subset of $X$ that is compatible with $R$ and that includes one of the equivalence classes, but not the other.
\end{definition}
\begin{lemma}[\mbox{\cite[Lemma 1, Chapter 37]{givant2008introduction}}]\label{l:Boole-eq-rel}
Let $X$ be a Boolean space and $R$ be an equivalence relation on $X$.
The quotient space $X/R$ is Boolean if and only if the relation $R$ is Boolean.
\end{lemma}
Similarly, we will identify a co-subobject $f \colon X \epi Y$ with a pair $(\sim_f,\mu_f)$, where $\sim_f$ is an equivalence relation and $\mu_f$ is a function from $X$ to $\NN$, as follows.
\begin{definition}\label{n:les-f}
Given an epic (= surjective) arrow $f\colon X \to Y$ of multisets, we set
\begin{equation*}
{\sim_f}=\{(x_1,x_2)\in X\times X\mid f(x_1)= f(x_2)\}.
\end{equation*}
and we denote by $\mu_f\colon X\to \NN$ the composite $\zeta_Y \circ f$.
\end{definition}
The idea is that, up to an iso, an epic arrow $f$ can be recovered from $(\sim_f,\mu_f)$.
In order to establish an inverse for the assignment $f \mapsto (\sim_f,\mu_f)$, we identify the characterising properties of $(\sim_f,\mu_f)$.
\begin{definition}\label{d:comp-strucutre}
Given a multiset $X$, we call \emph{multiset relation on $X$} a pair $(\sim, \mu)$, such that
\begin{enumerate}
\item \label{d:comp-strucutre:item1} $\sim$ is a Boolean relation on $X$,
\item \label{d:comp-strucutre:item2} $\mu\colon X\to \NN$ is a continuous function such that $\mu \leq \zeta$, 
\item \label{d:comp-strucutre:item3} for all $x,y\in X$, if $x\sim y$, then $\mu(x)=\mu(y)$.
\end{enumerate}
\end{definition}
\begin{lemma}
If $f\colon X \to Y$ is an epic arrow of multisets, then $(\sim_f, \mu_f)$ is a multiset relation on $X$.
\end{lemma}
\begin{proof}
\Cref{d:comp-strucutre:item1} in \cref{d:comp-strucutre} follows from \cref{l:Boole-eq-rel}.
\Cref{d:comp-strucutre:item2,d:comp-strucutre:item3} in \cref{d:comp-strucutre} are immediate.
\end{proof}
We will now illustrate how to recover an epic arrow $f$ from the multiset relation $(\sim, \mu)$.
\begin{definition}\label{r:sym}
Let $(X,\zeta_{X})$ be a multiset and $(\sim,\mu)$ be a multiset relation on $X$.
By \cref{l:Boole-eq-rel}, $X/{\sim}$ is a Boolean space.
Furthermore, define $\zeta_{\mu}\colon X/{\sim}\to \NN$ by setting $\zeta_{\mu}([x])=\mu(x)$.
It is immediate that this function is well defined and continuous, so we obtain an epic arrow of multisets $\pi_{\sim}\colon (X,\zeta_{X})\epi (X/{\sim},\zeta_{\mu})$.
\end{definition}
For $X$ a multiset, we let $\CS(X)$ denote the set of multiset relations on $X$.
We turn $\CS(X)$ into a partially ordered set by setting $({\sim_1}, \mu_1) \leq ({\sim_2}, \mu_2)$ if and only if ${\sim_1} \seq {\sim_2}$ and $\mu_1 \geq \mu_2$.

\begin{theorem}\label{p:bijection-Q-P}
The assignments
\begin{equation*}
\begin{tabular}{r c l c r c l}
$G\colon\Quot(X)$ & $ \longrightarrow$ & $\CS(X)$ &\hspace{1cm}  and \hspace{1cm}  &  $F\colon\CS(X)$ & $\longrightarrow$ & $\Quot(X)$\\
$(f\colon X\epi Y)$ & $ \longmapsto$ & ${(\sim_f,\mu_f)}$ &\hspace{1cm}    &  ${(\sim, \mu)}$ & $\longmapsto$ & $ ( \pi_{\sim}\colon (X,\zeta_{X})\epi (X/{\sim},\zeta_{\mu}))$
\end{tabular}
\end{equation*}
establish an isomorphism of partially ordered sets.
\end{theorem}
\begin{proof}
Let $f \colon X \to Y$ be an arrow of multisets. We want to prove that $f = F (G(f))=\pi_{\sim_{f}}$ in $\Quot(X)$.
It is sufficient to show that there is an isomorphism $\eps_f\colon X/{\sim_f}\to Y$ such that the following triangle commutes:
\[
\begin{tikzcd}
X \arrow{r}{\pi_{\sim_f}}\arrow[swap]{rd}{f}& X/{\sim_f}\arrow{d}{\eps_f}\\
& Y
\end{tikzcd}
\]
Define  $\eps_f$ by setting  $\eps_f([x])\df f(x)$. Observe that $x \sim_f y \Longleftrightarrow f(x) = f(y)$, for every $x, y \in X$.  Therefore, the function  $\eps_f$ is well defined and injective.  By definition, $f=\eps_f\circ \pi_{\sim_{f}}$.  Furthermore $\eps_f$ is surjective, for $f$ is so.
We now claim that $\eps_f$ is an arrow of multisets.
By definition of quotient topology, $\eps_f$ is continuous if and only if the composite $\eps_f\circ\pi_{\sim_{f}}$ is continuous; this holds because the composite is $f$.
By definition, for every $x\in X$, $\mu_{f}(x)=\zeta_Y(f(x))$; since the former is the denominator of $[x]$ and the latter is the denominator of the image of $[x]$ under $\eps_{F}$, the map $\eps_{f}$ preserves denominators.
Finally, since isomorphisms in $\MS$ are precisely the bijective denominator-preserving arrows (\cref{l:arrows-in-MS}), we conclude that $\eps_f$ is an isomorphism.

Let us now check that $(\sim, \mu)=G(F(\sim, \mu))=(\sim_{\pi_{\sim}}, \mu_{\pi_{\mu}})$.  By definition, for every $x, y \in X$,  $x \sim_{\pi_{\sim}} y$ if and only if $\pi_{\sim}(x) = \pi_{\sim}(y)$ if and only if $x\sim y$. Again a simple inspection of the definitions shows that, for every $x \in X$, 
\[
	\mu_{\pi_{\mu}}(x)=(\zeta_{\mu}\circ \pi_{\mu})(x) = \zeta_{\mu}(\pi_{\mu}(x))= \zeta_{\mu}([x]) = \mu(x),
\]
so that $\mu_{\pi_{\mu}} = \mu$.

We conclude the proof showing that the assignments in the statement preserve the order.
Suppose that $f\colon X\to Y$, $g\colon X\to Z$ are in $\Quot(X)$ and there exists an arrow $h \colon Y \to Z$  in $\MS$ such that  $h\circ f=g$. Then 
\[
	\mu_{g}=\zeta_{Z}\circ g=\zeta_{Z}\circ h\circ f\leq\zeta_{Y}\circ f=\zeta_{f},
\]
and if $f(x)=f(y)$ then also $g(x)=h(f(x))=h(f(y))=g(y)$.  Thus ${\mu_g}\leq{\mu_f}$ and ${\sim_f}\seq {\sim_{g}}$, which gives exactly $(\sim_f, \mu_f) \leq (\sim_g, \mu_g)$. 

Vice versa, suppose $(\sim_1, \mu_1), (\sim_2, \mu_2)$ are two multiset relations such that $(\sim_1, \mu_1) \leq (\sim_2, \mu_2)$. Thus, ${\sim_1}\seq {\sim_{2}}$ and ${\mu_2}\leq{\mu_1}$.  The former inclusion entails that the function $h\colon X/{\sim_{1}}\to X/{\sim_{2}}$ that maps $[x]_{\sim_{1}}$ to $[x]_{\sim_{2}}$ is well-defined and the latter entails that $h$ decreases denominators. Continuity follows from the obvious fact that $\pi_{2}=h\circ \pi_{1}$ and the definition of quotient topology. This concludes the proof.
\end{proof}

The representation of co-subobjects obtained in \cref{p:bijection-Q-P} is quite useful to study co-relations in $\MS$.
We illustrate this by characterising reflexive co-relations in $\MS$.  The next lemma re-proves \cref{l:char-co-refl} and shows that also the converse holds.

By \cref{p:bijection-Q-P}, we have a bijective correspondence between multisets relations on $X + X$ (i.e., elements of $\CS(X+X)$) and equivalence classes of co-relations on $X$ (i.e., elements of $\Quot(X+X)$).

Let us recall that a co-relation $(q_0,q_1) \colon X + X \to S$ is reflexive if the following diagram commutes:
\[
	\begin{tikzcd}
		{} 										& X+X \arrow[two heads,swap]{dl}{(q_0,q_1)} \arrow[two heads]{dr}{(1_X,1_X)}	& \\
		S \arrow[dashed]{rr}{\exists d}	& 																															& X
	\end{tikzcd}
\]
\begin{lemma}\label{l:char-refl}
	For any multiset $X$, a multiset relation $(\sim,\mu)$ on $X + X$ corresponds to a reflexive co-relation on $X$ if and only if the following hold:
\begin{enumerate}
\item $(x,i)\sim (y,j)\Longrightarrow x=y,$
\item $\mu(x,i)=\zeta_X(x).$
\end{enumerate}
\end{lemma}
\begin{proof}
	Let $(\approx,\eta)$ be the relational structure associated with $(1_X, 1_X)$, i.e.
\begin{enumerate}
\item $(x,i) \approx (y,j) \Longleftrightarrow x = y$,
\item	$\eta(x, i) = \zeta_X(x)$.
\end{enumerate}
	By \cref{p:bijection-Q-P}, $({\sim}, \mu)$ is reflexive if and only if $({\sim}, \mu) \leq (\approx, \eta)$, i.e.\ ${\sim} \seq {\approx}$ and $\mu \geq \eta$, by definition of the partial order on $\CS(X)$.
	The condition ${\sim} \seq {\approx}$ is precisely the condition $(x,i)\sim (y,j) \Longrightarrow x=y$.
	
Finally, $\eta = \zeta_{X + X}$ and by definition of multiset relation it holds that $\mu \leq \zeta_{X+X}$; thus $\mu \leq \eta$ and the condition $\eta \leq \mu$ is equivalent to $\eta = \mu$, i.e. $\mu(x,i) = \zeta_X(x)$.
\end{proof}

\section*{Acknowledgements}
The research of both authors was supported by the Italian Ministry of University and Research through the PRIN project  n.\ 20173WKCM5 \emph{Theory and applications of resource sensitive logics}.


\end{document}